\documentclass[11pt]{amsart}
\pdfoutput=1

\usepackage{amssymb,mathtools,stmaryrd}
\usepackage{enumitem}
\usepackage{microtype}

\usepackage{booktabs}
\usepackage{pgfplotstable}

\usepackage[margin=1in,marginparwidth=0.75in]{geometry}

\usepackage[numbers,compress]{natbib}

\usepackage{hyperref}
\hypersetup{
  pdftitle={A nonconforming primal hybrid finite element method for the two-dimensional vector Laplacian},
  pdfauthor={Mary Barker, Shuhao Cao, and Ari Stern},
  pdfsubject={MSC 2020: 65N30, 65N15, 35Q60},
  pdfkeywords={nonconforming finite element methods, hybridization, hydridizable discontinuous Galerkin methods, vector Laplacian, weighted Sobolev spaces},
  bookmarksopen=false
}

\usepackage[capitalize,nameinlink]{cleveref}

\makeatletter
\g@addto@macro\bfseries{\boldmath}
\makeatother

\theoremstyle{plain} 
\newtheorem{theorem}{Theorem}[section]
\newtheorem{lemma}[theorem]{Lemma}
\newtheorem{corollary}[theorem]{Corollary}

\newtheorem{conjecture}[theorem]{Conjecture}

\theoremstyle{definition}
\newtheorem{definition}[theorem]{Definition}

\theoremstyle{remark}
\newtheorem{remark}[theorem]{Remark}

\newcommand\llbrace{\{\!\!\{}
\newcommand\rrbrace{\}\!\!\}}

\begin{document}

\title[Nonconforming primal hybrid FEM for the 2D vector Laplacian]{A
  nonconforming primal hybrid finite element method for the
  two-dimensional vector Laplacian}

\author{Mary Barker}
\address{Public Health Sciences Division, Fred Hutchinson Cancer Center}
\email{marybarker103@gmail.com}
\thanks{MB was supported by the National Science Foundation under Grant No.\ DMS-1913272.}

\author{Shuhao Cao}
\address{Division of Computing, Analytics, and Mathematics, University of Missouri--Kansas City}
\email{scao@umkc.edu}
\thanks{SC was supported by the National Science Foundation under Grant No.\ DMS-2136075.}

\author{Ari Stern}
\address{Department of Mathematics, Washington University in St.~Louis}
\email{stern@wustl.edu}
\thanks{AS was supported by the National Science Foundation under Grant Nos.\ DMS-1913272 and DMS-2208551.}

\begin{abstract}
  We introduce a nonconforming hybrid finite element method for the
  two-dimensional vector Laplacian, based on a primal variational
  principle for which conforming methods are known to be
  inconsistent. Consistency is ensured using penalty terms similar to
  those used to stabilize hybridizable discontinuous Galerkin (HDG)
  methods, with a carefully chosen penalty parameter due to
  \citeauthor*{BrLiSu2007} [Math.\ Comp., 76 (2007),
  pp.~573--595]. Our method accommodates elements of arbitrarily high
  order and, like HDG methods, it may be implemented efficiently using
  static condensation. The lowest-order case recovers the
  $ P _1 $-nonconforming method of \citeauthor*{BrCuLiSu2008} [Numer.\
  Math., 109 (2008), pp.~509--533], and we show that higher-order
  convergence is achieved under appropriate regularity
  assumptions. The analysis makes novel use of a family of weighted
  Sobolev spaces, due to \citeauthor{Kondratiev1967}, for domains
  admitting corner singularities.
\end{abstract}

\maketitle

\section{Introduction}
\label{s:intro}

Given a bounded polygonal domain $ \Omega \subset \mathbb{R}^2 $,
$ f \in \bigl[ L ^2 (\Omega) \bigr] ^2 $, and
$ \alpha \in \mathbb{R} $, we consider the primal variational problem:
Find
$ u \in H ( \operatorname{div}; \Omega ) \cap \mathring{ H } (
\operatorname{curl} ; \Omega ) $ such that, for all
$v \in H ( \operatorname{div}; \Omega ) \cap \mathring{ H } (
\operatorname{curl} ; \Omega )$,
\begin{equation}
  \label{e:primal}
  ( \nabla \cdot u , \nabla \cdot v ) _\Omega + ( \nabla \times u , \nabla \times v ) _\Omega + \alpha ( u, v ) _\Omega = ( f, v ) _\Omega,
\end{equation}
where $ (\cdot , \cdot )_\Omega $ denotes the $ L^2 $ inner product on
$\Omega$. Here, we recall the familiar spaces
\begin{align*}
  H ( \operatorname{div} ; \Omega ) &\coloneqq \bigl\{ v \in \bigl[ L ^2 (\Omega) \bigr] ^2 : \nabla \cdot v \in L ^2 (\Omega) \bigr\},\\
  H ( \operatorname{curl} ; \Omega ) &\coloneqq \bigl\{ v \in \bigl[ L ^2 (\Omega) \bigr] ^2 : \nabla \times v \in L ^2 (\Omega) \bigr\},\\
  \mathring{H} ( \operatorname{curl} ; \Omega ) &\coloneqq \bigl\{ v \in H ( \operatorname{curl}; \Omega ) : v \times n = 0 \text{ on }  \partial \Omega \bigr\},
\end{align*}
where $n$ is the outer unit normal vector field on the boundary. The
strong form of \eqref{e:primal} is the boundary value problem
\begin{subequations}
  \label{e:strong}
  \begin{alignat}{2}
    - \nabla \nabla \cdot u + \nabla \times \nabla \times u + \alpha u &= f \quad &\text{in }&\Omega ,\label{e:strong_pde}\\
    u \times n &= 0 \quad &\text{on }& \partial \Omega ,\label{e:strong_bc_essential}\\
    \nabla \cdot u &= 0 \quad &\text{on }& \partial \Omega ,\label{e:strong_bc_natural}
  \end{alignat}
\end{subequations}
noting that
$ - \nabla \nabla \cdot {} + \nabla \times \nabla \times {} = - \Delta
$ is the negative vector Laplacian.

It is well established that conforming finite element methods for
\eqref{e:primal} have severe difficulties. For instance, a finite
element approximation that is both div- and curl-conforming will also
be $ H ^1 $-conforming, but when $ \Omega $ is non-convex,
$ \bigl[ H ^1 (\Omega) \bigr] ^2 \subsetneq H ( \operatorname{div} ;
\Omega ) \cap H ( \operatorname{curl} ; \Omega ) $ is a proper closed
subspace (see, e.g., \citep{BiSo1987,GiRa1986}).
Consequently, such a method
will fail to converge to solutions with reentrant corner
singularities. Even when $\Omega$ is convex, other problems may arise,
e.g., spurious modes for the eigenvalue problem \citep{BoDuGa1999,ArFaWi2010}. 
For this reason, the primal formulation \eqref{e:primal} is often avoided in favor of mixed
formulations of \eqref{e:strong} that are \emph{either} div- or
curl-conforming, but not both.

\citet*{BrCuLiSu2008} developed a $ P _1 $-\emph{nonconforming} primal
interior-penalty method that bypasses these difficulties. Let
$ \mathcal{T} _h $ be a conforming triangulation of $\Omega$, and let
$ \mathcal{E} _h $ denote the set of edges, partitioned into interior
edges $ \mathcal{E} _h ^\circ $ and boundary edges
$ \mathcal{E} _h ^\partial $. Denote the broken $ L ^2 $ inner
products
$ ( \cdot , \cdot ) _{ \mathcal{T} _h } \coloneqq \sum _{ K \in
  \mathcal{T} _h } ( \cdot , \cdot ) _K $,
$ \langle \cdot , \cdot \rangle _{ \mathcal{E} _h ^\circ } \coloneqq
\sum _{ e \in \mathcal{E} _h ^\circ } \langle \cdot , \cdot \rangle _e
$, and
$ \langle \cdot , \cdot \rangle _{ \mathcal{E} _h ^\partial }
\coloneqq \sum _{ e \in \mathcal{E} _h ^\partial } \langle \cdot ,
\cdot \rangle _e $. The method of \citep{BrCuLiSu2008} is based on the
variational problem
\begin{multline}
  \label{e:brenner}
  ( \nabla \cdot u _h , \nabla \cdot v _h ) _{ \mathcal{T} _h } + (
  \nabla \times u _h , \nabla \times v _h ) _{ \mathcal{T} _h } +
  \alpha ( u _h , v _h ) _{ \mathcal{T} _h } + \bigl\langle \gamma
  \llbracket u _h \rrbracket , \llbracket v _h \rrbracket \bigr\rangle
  _{ \mathcal{E} _h ^\circ } + \langle \gamma
  u _h \times n , v _h \times n \rangle _{ \mathcal{E} _h ^\partial } \\
  = ( f , v _h ) _{ \mathcal{T} _h },
\end{multline}
where $ \gamma \rvert _e \coloneqq \gamma _e > 0 $ is a penalty
parameter on each $ e \in \mathcal{E} _h $, to be detailed further in
\cref{s:method}, and $ \llbracket \cdot \rrbracket $ is the jump in
both tangential and normal components across an interior edge. (See
also \citet*{BrLiSu2007,BrLiSu2008,BrLiSu2009a,BrLiSu2009b} for
related work on curl-curl source problems and eigenproblems arising in
Maxwell's equations.) In \citep{BrCuLiSu2008}, $ u _h $ and $ v _h $
are linear vector fields continuous at the midpoint of each
$ e \in \mathcal{E} _h ^\circ $ (i.e., both components live in the
$ P _1 $-nonconforming space of \citet{CrRa1973}) whose tangential
components vanish at the midpoint of each
$ e \in \mathcal{E} _h ^\partial $. \citet{BrSu2009} later developed a
quadratic nonconforming element for this problem and conjectured that
it could be generalized to higher degree, as well as to dimension
three. The two-dimensional conjecture was subsequently proved by
\citet{Mirebeau2012}, who also gave a counterexample to the
three-dimensional case.  However, for $ k > 1 $, the order-$k$
elements are not simply $ P _k $ vector fields: they are enriched by
additional vector fields up to degree $ 2 k -1 $ that are gradients of
harmonic polynomials.

In this paper, we present a three-field primal hybridization of
\eqref{e:brenner} in the following form: Find
$ ( u _h , p _h , \hat{u} _h ) \in V _h \times Q _h \times \hat{V} _h
$ such that, for all
$ ( v _h , q _h , \hat{v} _h ) \in V _h \times Q _h \times \hat{V} _h
$,
\begin{subequations}
  \label{e:hybrid}
  \begin{align}
    ( \nabla \cdot u _h , \nabla \cdot v _h ) _{ \mathcal{T} _h } + ( \nabla \times u _h , \nabla \times v _h ) _{ \mathcal{T} _h } +\alpha ( u _h , v _h ) _{ \mathcal{T} _h } + \langle \hat{p} _h , v _h \rangle _{ \partial \mathcal{T} _h } &= ( f , v _h ) _{ \mathcal{T} _h } , \label{e:hybrid_v}\\
    \langle u _h - \hat{u} _h , q _h \rangle _{ \partial \mathcal{T} _h } &= 0 , \label{e:hybrid_q}\\
    \langle \hat{p} _h , \hat{v} _h \rangle _{ \partial \mathcal{T} _h } &= 0 , \label{e:hybrid_vhat}
  \end{align}
\end{subequations}
where $ \hat{p} _h \coloneqq p _h + \gamma ( u _h - \hat{u} _h ) $ and
$ \langle \cdot , \cdot \rangle _{ \partial \mathcal{T} _h } \coloneqq
\sum _{ K \in \mathcal{T} _h } \langle \cdot , \cdot \rangle _{
  \partial K } $. With appropriately chosen finite element spaces, as detailed in \cref{s:method}, this method has the following properties:
\begin{itemize}
\item The lowest-order case is a hybridization of the method of
  \citet{BrCuLiSu2008}.

\item Arbitrarily high order may be obtained using standard polynomial
  finite elements. The more exotic Brenner--Sung--Mirebeau spaces and
  projections play a crucial role in the analysis but are not needed
  for implementation.

\item As with HDG methods \citep{CoGoLa2009}, the hybrid formulation
  enables efficient local assembly and static condensation, where
  $ u _h $ and $ p _h $ may be eliminated to solve a smaller global
  system involving only the approximate trace $ \hat{u} _h $.
\end{itemize}
In addition to these contributions, we also present a novel error
analysis using weighted Sobolev spaces, cf.~\citet{CoDa2002}. 
This approach allows us to obtain error estimates on domains admitting
corner singularities, without imposing the mesh-grading conditions on
$ \mathcal{T} _h $ required by \citep{BrCuLiSu2008}.

The paper is organized as follows. In \cref{s:method}, we describe the
method and discuss its fundamental properties. Next, in
\cref{s:analysis}, we present the error analysis of the
method. Finally, in \cref{s:experiments}, we present the results of
numerical experiments, which demonstrate and confirm the analytically
obtained convergence results.

\section{The hybrid method}
\label{s:method}

\subsection{Description of the method}
\label{s:description}

The proposed method is based on the variational problem
\eqref{e:hybrid}, using the following finite element spaces. Given a
positive integer $k$, define
\begin{equation*}
  V _h \coloneqq \prod _{ K \in \mathcal{T} _h } \bigl[ P _{ 2 k -1 } (K) \bigr] ^2 , \qquad Q _h \coloneqq \prod _{ K \in \mathcal{T} _h } \prod _{ e \subset \partial K } \bigl[ P _{ k -1 } (e) \bigr] ^2 .
\end{equation*}
These are ``broken'' finite element spaces, with no inter-element
continuity or boundary conditions imposed, so vector fields in these
spaces are generally double-valued at interior edges. To (weakly)
impose inter-element continuity and boundary conditions, we define
\begin{equation*}
  \hat{V} _h \coloneqq \biggl\{ \hat{v} _h \in \prod _{ e \in \mathcal{E} _h } \bigl[ P _{ 2 k -1 } (e) \bigr] ^2 : \hat{v} _h \times n = 0 \text{ on } \mathcal{E} _h ^\partial \biggr\} ,
\end{equation*}
whose elements are single-valued on edges.

The extra variables $ p _h $ and $ \hat{u} _h $, and their role in the
variational problem, may be understood as follows. From
\eqref{e:hybrid_q}, we see that $ p _h $ acts as a Lagrange
multiplier, constraining the degree $ \leq k -1 $ moments of $ u _h $
and $ \hat{u} _h $ to agree on $ \mathcal{E} _h $. Consequently,
$ u _h $ satisfies weak inter-element continuity and boundary
conditions, and $ \hat{u} _h $ may be seen as an approximate trace of
$u$. Next, on each $ K \in \mathcal{T} _h $, taking the inner product
of the strong form \eqref{e:strong_pde} with
$ v \in H ( \operatorname{div} ; K ) \cap H ( \operatorname{curl} ; K
) $ and integrating by parts implies that the solution to
\eqref{e:primal} satisfies
\begin{equation}
  \label{e:ibp}
  ( \nabla \cdot u , \nabla \cdot v ) _K + ( \nabla \times u , \nabla \times v ) _K + \alpha ( u , v ) _K - \langle \nabla \cdot u , v \cdot n \rangle _{ \partial K } + \langle \nabla \times u , v \times n \rangle _{ \partial K } = ( f, v ) _K. 
\end{equation}
Comparing with \eqref{e:hybrid_v} and writing
$ \langle \hat{p} _h , v _h \rangle _{ \partial K } = \langle \hat{p}
_h \cdot n , v _h \cdot n \rangle _{ \partial K } + \langle \hat{p} _h
\times n , v _h \times n \rangle _{ \partial K } $, it follows that
$ \hat{p} _h \cdot n \rvert _{ \partial K } $ and
$ \hat{p} _h \times n \rvert _{ \partial K } $ can be seen as
approximating $ - \nabla \cdot u \rvert _{ \partial K } $ and
$ \nabla \times u \rvert _{ \partial K } $, respectively. Finally,
\eqref{e:hybrid_vhat} shows that $ \hat{u} _h $ also acts as a
Lagrange multiplier, constraining $ \hat{p} _h \cdot n $ and
$ \hat{p} _h \times n $ to be single-valued on interior edges and
$ \hat{p} _h \cdot n = 0 $ on boundary edges. The latter may be seen
as an approximation of the natural boundary condition
\eqref{e:strong_bc_natural}.

To ensure convergence of the method for solutions with corner
singularities, the penalty $ \gamma $ must be chosen carefully. Here,
we recall the penalty used by \citet{BrCuLiSu2008}, which is the same
one that we will use.  Denote the corners of $\Omega$ by
$ c _1 , \ldots, c _L $, and let
$ r _\ell (x) \coloneqq \lvert x - c_\ell \rvert $ be the distance
from $ x \in \Omega $ to each corner. Given a multi-exponent
$ \lambda = ( \lambda _1 , \ldots , \lambda _L ) $, we denote
$ r ^\lambda \coloneqq \prod _{ \ell=1 } ^L r _\ell ^{ \lambda _\ell }
$. Now, at each corner $ c_\ell $, with interior angle
$ \omega _\ell $, choose a parameter $ \mu _\ell $ such that
\begin{equation*}
  \mu _\ell = 1 \quad \text{if } \omega _\ell \leq \frac{ \pi }{ 2 } , \qquad \mu _\ell < \frac{ \pi }{ 2 \omega _\ell } \quad \text{if } \omega _\ell > \frac{ \pi }{ 2 } ,
\end{equation*}
and $ \mu \coloneqq ( \mu _1 , \ldots, \mu _L ) $. For each
$ e \in \mathcal{E} _h $, whose midpoint is denoted $m _e$, we then
define
\begin{equation*}
  \Phi _\mu (e) \coloneqq r ^{ 1 - \mu } (m _e) = \prod _{ \ell = 1 } ^L \lvert m _e - c _\ell \rvert ^{ 1 - \mu _\ell } .
\end{equation*}
Finally, the penalty parameter on $e$ is taken to be
\begin{equation*}
  \gamma _e \coloneqq \frac{ \bigl[ \Phi _\mu (e) \bigr] ^2 }{ \lvert e \rvert } ,
\end{equation*}
where $ \lvert e \rvert $ is the length of $e$.  This ensures that
$ \gamma _e \sim 1 / \lvert e \rvert $ away from corners, while being
appropriately weakened near corners to allow convergence to singular
solutions, as we will see in \cref{s:analysis}.

\subsection{The Brenner--Sung--Mirebeau element and projection}

We now recall the nonconforming finite element developed in
\citet{BrSu2009} and \citet{Mirebeau2012}, which we call the
Brenner--Sung--Mirebeau (BSM) element. While it is not used to
implement the method described above, this element and its associated
projection play an important role in the numerical analysis of the
method---and will also make clear why we have taken polynomial spaces
of degrees $ 2 k -1 $ and $ k -1 $.

\begin{definition}
  Given a positive integer $k$, define the
  \emph{Brenner--Sung--Mirebeau (BSM) space} on a triangle
  $K \subset \mathbb{R}^2 $ to be
  \begin{equation*}
    \textit{BSM} _k (K) \coloneqq \bigl[ P _k (K) \bigr] ^2 \oplus \nabla H _{ k + 2 } (K) \oplus \cdots \oplus \nabla H _{ 2 k } (K) ,
  \end{equation*}
  where $ H _j (K) $ is the space of homogeneous harmonic polynomials
  of degree $ j $ on $K$. (By \emph{harmonic}, we mean having
  vanishing Laplacian.)
\end{definition}

We immediately see that
$ \bigl[ P _k (K) \bigr] ^2 \subset \textit{BSM} _k (K) \subset \bigl[ P _{ 2
  k -1 } (K) \bigr] ^2 $, with equality if and only if $ k = 1
$. Indeed, since $ \dim H _j (K) = 2 $ for each $ j $, it
follows that $ \dim \textit{BSM} _k (K) = k ( k + 5 )
$. \citet{BrSu2009} conjectured, and \citet{Mirebeau2012} proved, that
an element $ v _h \in \textit{BSM} _k (K) $ is uniquely determined by
the $ k ( k + 5 ) $ degrees of freedom
\begin{equation*}
  \langle v _h , q _h \rangle _{ \partial K } \,\text{ for }\, q _h \in \prod _{ e \subset \partial K } \bigl[ P _{ k -1 } (e) \bigr] ^2 , \qquad ( v _h , w _h ) _K \,\text{ for }\, w _h \in \bigl[ P _{ k - 2 } (K) \bigr] ^2 .
\end{equation*}
Moreover, the canonical interpolation using these degrees of freedom
naturally defines a projection
$ \Pi _h \colon \bigl[ H ^\sigma (K) \bigr] ^2 \rightarrow \textit{BSM} _k
(K) $, for any $ \sigma > \frac{1}{2} $, such that
\begin{equation}
  \label{e:bsm_projection}
  \langle \Pi _h v , q _h \rangle _{ \partial K } = \langle v, q _h \rangle _{ \partial K } , \qquad   ( \Pi _h v , w _h ) _K = ( v , w _h ) _K , 
\end{equation}
for all $ q _h $ and $ w _h $ as above. Letting
$ P _h \colon L ^2 (K) \rightarrow P _{ k - 1 } (K) $ be the
$ L ^2 $-orthogonal projection for scalar fields, we obtain the
following commuting-projection property; the proof is basically
identical to that in \citet{BrSu2009}.

\begin{lemma}
  \label{l:commuting_projection}
  For all $ v \in \bigl[ H ^\sigma (K) \bigr] ^2 $ with $ \sigma > \frac{1}{2} $
  such that $ \nabla \cdot v, \nabla \times v \in L ^2 (K) $, we have
  \begin{equation*}
    \nabla \cdot \Pi _h v = P _h ( \nabla \cdot v ) , \qquad 
    \nabla \times \Pi _h v = P _h ( \nabla \times v ) .
  \end{equation*}
\end{lemma}

\begin{proof}
  For all $ \phi _h \in P _{ k - 1 } (K) $, integrating by parts using
  the divergence theorem gives
  \begin{equation*}
    \bigl( \nabla \cdot ( v - \Pi _h v ) , \phi _h \bigr) _K = \langle  v - \Pi _h v ,  \phi _h  n \rangle _{ \partial K } - ( v - \Pi _h v , \nabla \phi _h ) _K = 0 ,
  \end{equation*}
  by \eqref{e:bsm_projection} with $ q _h = \phi _h n $ and
  $ w _h = \nabla \phi _h $. This proves the first equality; the proof
  of the second is essentially the same, using Green's theorem instead
  of the divergence theorem.
\end{proof}

The solution to \eqref{e:primal} satisfies the hypotheses of this
lemma on each $ K \in \mathcal{T} _h $, as a result of the regularity
theory discussed in \cref{s:weighted}.

\subsection{Equivalence to reduced methods with jump terms}

We next show that the three-field hybrid method described in
\cref{s:description} may be reduced to a two-field or one-field method
with jump terms. The coupling introduced by the jump terms prevents
static condensation, so we generally prefer the three-field
formulation for implementation. However, these reduced formulations
will be useful analytically, and will help in relating our method to
that of \citet{BrCuLiSu2008}.\footnote{This approach to reduction was
  inspired by some recent papers on ``extended Galerkin'' methods
  \citep{HoWuXu2021,HoLiXu2022}.}

First, we introduce notation and definitions for the average and jump
of a vector field across interior edges. Suppose
$ e \in \mathcal{E} _h ^\circ $ is an interior edge shared by two
triangles, $ K ^+ $ and $ K ^- $, and let $ n ^\pm $ denote the unit
normal to $e$ pointing outward from $ K ^\pm $.  If a vector field $w$
takes values $ w ^\pm $ on the $ K^\pm $ sides of $e$, we define the
average and jump of $w$ at $e$ to be
\begin{equation*}
  \llbrace w \rrbrace _e \coloneqq \frac{1}{2} ( w ^+ + w ^- ), \qquad \llbracket w  \rrbracket _e \coloneqq w ^+ \otimes n ^+ + w ^- \otimes n ^- ,
\end{equation*}
where $w\otimes n \coloneqq w n^\top$ is the outer product. It is
straightforward to see that the $i$-th row of
$ \llbracket w \rrbracket _e $ is the transpose of the usual scalar
jump $ \llbracket w _i \rrbracket _e = w _i ^+ n ^+ + w _i ^- n ^- $
for $ i = 1, 2 $. This definition of $ \llbracket \cdot \rrbracket $
encodes the jump in both tangential and normal directions, without
requiring a global orientation of the edges.  It is then easily
verified that the
$ \langle \cdot , \cdot \rangle _{ \partial \mathcal{T} _h } $ inner
product of vector fields (just as for scalar fields) may be expanded
as
\begin{equation}
  \label{e:edge_identity}
  \langle w , v \rangle _{ \partial \mathcal{T} _h } = 2 \bigl\langle \llbrace w \rrbrace , \llbrace v \rrbrace \bigr\rangle _{ \mathcal{E} _h ^\circ } + \frac{1}{2} \bigl\langle \llbracket w \rrbracket , \llbracket v \rrbracket \bigr\rangle _{ \mathcal{E} _h ^\circ } + \langle w, v \rangle _{ \mathcal{E} _h ^\partial } ,
\end{equation}
where the inner product of the matrix-valued jumps is taken in
the Frobenius sense. Since functions are single-valued on boundary
edges, we leave average and jump undefined on
$ \mathcal{E} _h ^\partial $.

\begin{lemma}
  \label{l:avg}
  If $ ( u _h , p _h , \hat{u} _h ) $ satisfies
  \eqref{e:hybrid_q}--\eqref{e:hybrid_vhat}, then
  \begin{alignat*}{3}
    \llbrace u _h \rrbrace &= \hat{u} _h ,\qquad & \llbrace p _h \rrbrace &= 0 , \quad &\text{on } &\mathcal{E} _h ^\circ,\\
    u _h \cdot n &= \hat{u} _h \cdot n , \qquad & p _h \cdot n &= 0 ,
    \quad &\text{on } & \mathcal{E} _h ^\partial .
  \end{alignat*}
\end{lemma}

\begin{proof}
  Let $ q _h = \llbrace p _h \rrbrace $ on $ \mathcal{E} _h ^\circ $
  and $ q _h = ( p _h \cdot n ) n $ on $ \mathcal{E} _h ^\partial
  $. Taking this as the test function in \eqref{e:hybrid_q} and
  applying the identity \eqref{e:edge_identity} gives
  \begin{equation}
    \label{e:hybrid_q_edges}
    2 \bigl\langle \llbrace u _h \rrbrace - \hat{u} _h , \llbrace p _h \rrbrace \bigr\rangle _{ \mathcal{E} _h ^\circ } + \bigl\langle ( u _h - \hat{u} _h ) \cdot n , p _h \cdot n \bigr\rangle _{ \mathcal{E} _h ^\partial } = 0 .
  \end{equation}
  There are no interior jump-jump or tangential boundary terms, since
  this choice of $ q _h $ has $ \llbracket q _h \rrbracket = 0 $ on
  $ \mathcal{E} _h ^\circ $ and $ q _h \times n = 0 $ on
  $ \mathcal{E} _h ^\partial $. Similarly,
  $ \llbracket \hat{v} _h \rrbracket = 0 $ on
  $ \mathcal{E} _h ^\circ $ and $ \hat{v} _h \times n = 0 $ on
  $ \mathcal{E} _h ^\partial $ for all $ \hat{v} _h \in \hat{V} _h $,
  so \eqref{e:hybrid_vhat} may be rewritten as
  \begin{equation}
    \label{e:hybrid_vhat_edges}
    2 \bigl\langle \llbrace p _h \rrbrace + \gamma \bigl( \llbrace u _h \rrbrace - \hat{u} _h \bigr) , \hat{v} _h \bigr\rangle _{ \mathcal{E} _h ^\circ } + \bigl\langle  p _h \cdot n + \gamma ( u _h - \hat{u} _h ) \cdot n , \hat{v} _h \cdot n \bigr\rangle _{ \mathcal{E} _h ^\partial } = 0 .
  \end{equation}
  Now, take $ \hat{v} _h = \llbrace u _h \rrbrace - \hat{u} _h $ on
  $ \mathcal{E} _h ^\circ $ and
  $ \hat{v} _h \cdot n = ( u _h - \hat{u} _h ) \cdot n $ on
  $ \mathcal{E} _h ^\partial $. The terms involving $ p _h $ vanish by
  \eqref{e:hybrid_q_edges}, leaving
  \begin{equation*}
    2 \bigl\langle \gamma \bigl( \llbrace u _h \rrbrace - \hat{u} _h \bigr), \llbrace u _h \rrbrace - \hat{u} _h \bigr\rangle _{ \mathcal{E} _h ^\circ } + \bigl\langle \gamma ( u _h - \hat{u} _h ) \cdot n , ( u _h - \hat{u} _h ) \cdot n \bigr\rangle _{ \mathcal{E} _h ^\partial } = 0 .
  \end{equation*}
  Since $ \gamma > 0 $, it follows that
  $ \llbrace u _h \rrbrace = \hat{u} _h $ on $ \mathcal{E} _h ^\circ $
  and $ u _h \cdot n = \hat{u} _h \cdot n $ on
  $ \mathcal{E} _h ^\partial $, as claimed. 
  Finally, substituting these equalities into \eqref{e:hybrid_vhat_edges} gives
  \begin{equation*}
    2 \bigl\langle \llbrace p _h \rrbrace , \hat{v} _h \bigr\rangle _{ \mathcal{E} _h ^\circ } + \langle p _h \cdot n , \hat{v} _h \cdot n \rangle _{ \mathcal{E} _h ^\partial } = 0 ,
  \end{equation*}
  so taking $ \hat{v} _h = \llbrace p _h \rrbrace $ on
  $ \mathcal{E} _h ^\circ $ and $ \hat{v} _h \cdot n = p _h \cdot n $
  on $ \mathcal{E} _h ^\partial $ completes the proof.
\end{proof}

\begin{remark}
  Note that $ \llbrace w \rrbrace _e = 0 $ can be rewritten as
  $ w ^+ = - w ^- $. Since the outer normals satisfy $ n ^+ = - n^- $,
  it follows that $ w ^+ \times n ^+ = w ^- \times n ^- $ and
  $ w ^+ \cdot n ^+ = w ^- \cdot n ^- $, i.e., the tangential and
  normal components of $w$ agree on both sides of $e$. Thus,
  \cref{l:avg} says that the tangential and normal components of
  $ p _h $ and $ u _h - \hat{u} _h $ are single-valued, with normal
  components vanishing on boundary edges. In particular, the same is
  therefore true of $ \hat{p} _h $, as previously remarked in
  \cref{s:description}.
\end{remark}  

Using \cref{l:avg} and the identity \eqref{e:edge_identity}, observe
that the edge terms in \eqref{e:hybrid_v} reduce to
\begin{align*}
  \langle \hat{p} _h , v _h \rangle _{ \partial \mathcal{T} _h }
  &= \frac{1}{2} \bigl\langle \llbracket \hat{p} _h \rrbracket , \llbracket v _h \rrbracket \bigr\rangle _{ \mathcal{E} _h ^\circ } + \langle \hat{p} _h \times n , v _h \times n \rangle _{ \mathcal{E} _h ^\partial } \\
  &= \frac{1}{2} \bigl\langle \llbracket p _h + \gamma u _h \rrbracket , \llbracket v _h \rrbracket \bigr\rangle _{ \mathcal{E} _h ^\circ } + \bigl\langle ( p _h + \gamma u _h ) \times n , v _h \times n \bigr\rangle _{ \mathcal{E} _h ^\partial } ,
\end{align*}
for all $ v _h \in V _h $.  Similarly, the edge terms in
\eqref{e:hybrid_q} reduce to
\begin{equation*}
  \langle u _h - \hat{u} _h , q _h \rangle _{ \partial \mathcal{T} _h } = \frac{1}{2} \bigl\langle \llbracket u _h \rrbracket , \llbracket q _h \rrbracket \bigr\rangle _{ \mathcal{E} _h ^\circ } + \langle u _h \times n , q _h \times n \rangle _{ \mathcal{E} _h ^\partial } ,
\end{equation*}
for all $ q _h \in Q _h $. This allows us to eliminate $ \hat{u} _h $
and the equation \eqref{e:hybrid_vhat} from the variational problem.
A two-field reduced formulation is defined as follows. Let
\begin{equation*}
  \mathring{ Q } _h \coloneqq \bigl\{ q _h \in Q _h : \llbrace q _h \rrbrace = 0 \text{ on } \mathcal{E} _h ^\circ \text{ and } q _h \cdot n = 0 \text{ on } \mathcal{E} _h ^\partial \bigr\} ,
\end{equation*}
and define the bilinear forms
$ a _h \colon V _h \times V _h \rightarrow \mathbb{R} $ and
$ b _h \colon V _h \times \mathring{ Q } _h \rightarrow \mathbb{R} $
by
\begin{align*}
  a _h ( u _h , v _h )
  &\coloneqq \begin{multlined}[t]
    ( \nabla \cdot u _h , \nabla \cdot v _h ) _{ \mathcal{T} _h } + ( \nabla \times u _h , \nabla \times v _h ) _{ \mathcal{T} _h } + \alpha ( u _h , v _h ) _{ \mathcal{T} _h } \\
    + \frac{1}{2} \bigl\langle \gamma \llbracket u _h \rrbracket , \llbracket v _h \rrbracket \bigr\rangle _{ \mathcal{E} _h ^\circ } + \langle \gamma u _h \times n , v _h \times n \rangle _{ \mathcal{E} _h ^\partial },
  \end{multlined}
  \\
  b _h ( v _h , q _h )
  &\coloneqq \frac{1}{2} \bigl\langle \llbracket v _h \rrbracket , \llbracket q _h \rrbracket \bigr\rangle _{ \mathcal{E} _h ^\circ } + \langle v _h \times n , q _h \times n \rangle _{ \mathcal{E} _h ^\partial } .
\end{align*}
We then consider the problem: Find
$ ( u _h , p _h ) \in V _h \times \mathring{ Q } _h $ such that, for
all $ ( v _h , q _h ) \in V _h \times \mathring{ Q } _h $,
\begin{subequations}
  \label{e:saddle}
  \begin{align}
    a _h ( u _h , v _h ) + b _h ( v _h, p _h ) &= (f, v _h ) _{ \mathcal{T} _h }, \label{e:saddle_v}\\
    b _h ( u _h , q _h ) &= 0 . \label{e:saddle_q}
  \end{align}
\end{subequations}
This resembles a standard two-field hybrid method in saddle-point
form, where $ \mathring{ Q } _h $ is the space of Lagrange
multipliers. Compare with the nonconforming hybrid method of
\citet{RaTh1977} for the scalar Poisson equation.

Finally, we may reduce even further to a one-field formulation on
\begin{equation*}
  \mathring{ V } _h \coloneqq \bigl\{ v _h \in V _h : b _h ( v _h , q _h ) = 0 \text{ for all } q _h \in \mathring{ Q } _h \bigr\} ,
\end{equation*}
consisting of degree-$ ( 2 k -1 ) $ vector fields whose degree
$ \leq k -1 $ moments are continuous on $ \mathcal{E} _h ^\circ $ and
have vanishing tangential component on $ \mathcal{E} _h ^\partial $.
We then consider the problem: Find $ u _h \in \mathring{ V } _h $ such
that, for all $ v _h \in \mathring{ V } _h $,
\begin{equation}
  \label{e:one-field}
  a _h ( u _h , v _h ) = ( f , v _h ) _{ \mathcal{T} _h } .
\end{equation}
This is precisely \eqref{e:brenner}, modulo a constant factor of
$ \frac{1}{2} $ for the penalty on interior edges, and the
lowest-order case $ k = 1 $ recovers the method of
\citet{BrCuLiSu2008}.

We have thus shown the equivalence of the three-field, two-field, and
one-field formulations, which we now state as a lemma.

\begin{lemma}
  \label{l:reduced}
  The following are equivalent:
  \begin{enumerate}[label=(\roman*)]
  \item $ ( u _h , p _h , \hat{u} _h ) \in V _h \times Q _h \times \hat{V} _h $ is a solution to \eqref{e:hybrid};

  \item $ ( u _h , p _h ) \in V _h \times \mathring{ Q } _h $ is a
    solution to \eqref{e:saddle},
    $ \hat{u} _h = \llbrace u _h \rrbrace $ on
    $ \mathcal{E} _h ^\circ $, and
    $ \hat{u} _h \cdot n = u _h \cdot n $ on
    $ \mathcal{E} _h ^\partial $; \label{i:two-field}

  \item $ u _h \in \mathring{ V } _h $ is a solution to
    \eqref{e:one-field}, $ p _h $ satisfies \eqref{e:saddle_v} for all
    $ v _h \in V _h $, and $ \hat{u} _h $ is as in \ref{i:two-field}.
  \end{enumerate} 
\end{lemma}

\subsection{Existence/uniqueness and static condensation}
\label{s:wp_sc}

The problem \eqref{e:primal} is well-posed if and only if $\alpha$ is
not an eigenvalue of the vector Laplacian on
$ H ( \operatorname{div}; \Omega ) \cap \mathring{ H } (
\operatorname{curl} ; \Omega ) $. In particular, the bilinear form
\begin{equation*}
  a ( u , v ) \coloneqq ( \nabla \cdot u , \nabla \cdot v ) _\Omega + ( \nabla \times u , \nabla \times v ) _\Omega + \alpha ( u , v ) _\Omega 
\end{equation*}
on
$ H ( \operatorname{div}; \Omega ) \cap \mathring{ H } (
\operatorname{curl} ; \Omega ) $ is coercive if $ \alpha > 0 $, and if
the complement of $\Omega$ is connected (e.g., $\Omega$ is simply
connected), then it is also coercive for $ \alpha = 0 $ by
Friedrichs's inequality (cf.~\citet{Monk2003}).

We are now ready to prove our main result on existence and uniqueness
for the hybrid method.

\begin{theorem}
  \label{t:well-posed}
  For the problems \eqref{e:hybrid}, \eqref{e:saddle}, and
  \eqref{e:one-field}, existence and uniqueness of solutions
  holds---or fails to hold---simultaneously for all three. In
  particular, all three are uniquely solvable if $ \alpha > 0 $, and
  if the complement of $\Omega$ is connected, then this is also true
  for $ \alpha = 0 $.
\end{theorem}

\begin{proof}
  By \cref{l:reduced}, unique solvability of \eqref{e:hybrid} is
  equivalent to that of \eqref{e:saddle}, since $ \hat{u} _h $ is
  uniquely determined by $ u _h $, so it suffices to show equivalence
  of \eqref{e:saddle} and \eqref{e:one-field}.  Using classic
  saddle-point theory (cf.~\citet[Theorem 3.2.1]{BoBrFo2013}),
  \eqref{e:saddle} is uniquely solvable if and only if
  $ v _h \mapsto b _h ( v _h , \cdot ) $ is surjective and the
  restriction of $ a _h ( \cdot , \cdot ) $ to its kernel is an
  isomorphism. The isomorphism-on-the-kernel condition is precisely
  the unique solvability of \eqref{e:one-field}, so it remains to show
  that the surjectivity condition holds.

  In fact, we will show something slightly stronger, which is
  surjectivity of the map $ B _h \colon V _h \rightarrow Q _h ^\ast $,
  $ \langle B _h v _h , q _h \rangle _{ Q _h ^\ast \times Q _h }
  \coloneqq \langle v _h , q _h \rangle _{ \partial \mathcal{T} _h }
  $, which agrees with $ b _h ( v _h , q _h ) $ when
  $ q _h \in \mathring{ Q } _h $ by \eqref{e:edge_identity}. Given
  $ q _h \in Q _h $, on each $ K \in \mathcal{T} _h $ there exists
  $ v _h \rvert _K \in \textit{BSM} _k (K) $ whose degree
  $ \leq k -1 $ moments agree with $ q _h $ on $ \partial K
  $. Combining these into $ v _h \in V _h $, it follows that
  $ \langle v _h , q _h \rangle _{ \partial \mathcal{T} _h } = \langle
  q _h , q _h \rangle _{ \partial \mathcal{T} _h } $, which is
  strictly positive whenever $ q _h $ is nonzero. Hence, the transpose
  of $ B _h $ is injective, so $ B _h $ is surjective.

  Finally, if $ \alpha > 0 $, then $ a _h ( \cdot , \cdot ) $ is
  clearly positive-definite, so \eqref{e:one-field} is uniquely
  solvable. If $ \alpha = 0 $, then $ a _h ( \cdot , \cdot ) $ is
  generally only positive-semidefinite. However, observe that
  $ a _h ( u _h , u _h ) = 0 $ implies
  $ \llbracket u _h \rrbracket = 0 $ on $ \mathcal{E} _h ^\circ $ and
  $ u _h \times n = 0 $ on $ \mathcal{E} _h ^\partial $, so
  $ u _h \in H ( \operatorname{div} ; \Omega ) \cap \mathring{ H } (
  \operatorname{curl} ; \Omega ) $ with $ a ( u _h , u _h ) = 0 $. If
  the complement of $\Omega$ is connected, then
  $ a ( \cdot , \cdot ) $ is coercive, so $ u _h = 0 $, and thus
  $ a _h ( \cdot , \cdot ) $ is positive-definite.
\end{proof}

Next, we discuss the static condensation of the hybrid method, which
eliminates the spaces $ V _h $ and $ Q _h $ from \eqref{e:hybrid} to
obtain a smaller global variational problem on $ \hat{V} _h $
alone. We take a similar approach to that used for HDG methods in
\citet*{CoGoLa2009}. Observe that, given $ \hat{u} _h $ and $f$,
\eqref{e:hybrid_v}--\eqref{e:hybrid_q} state that
$ ( u _h , p _h ) \in V _h \times Q _h $ solves the local problems
\begin{align*}
  ( \nabla \cdot u _h , \nabla \cdot v _h ) _K + ( \nabla \times u _h , \nabla \times v _h ) _K + \alpha ( u _h , v _h ) _K + \langle p _h + \gamma u _h , v _h \rangle _{ \partial K } &= ( f, v _h ) _K + \langle \gamma \hat{u} _h , v _h \rangle _{ \partial K } ,\\
  \langle u _h , q _h \rangle _{ \partial K } &= \langle \hat{u} _h , q _h \rangle _{ \partial K } ,
\end{align*}
for all $ K \in \mathcal{T} _h $ and
$ ( v _h , q _h ) \in V _h \times Q _h $. On each
$ K \in \mathcal{T} _h $, define the local bilinear forms
\begin{align*}
  a _K ( u _h , v _h ) &\coloneqq ( \nabla \cdot u _h , \nabla \cdot v _h ) _K + ( \nabla \times u _h , \nabla \times v _h ) _K + \alpha ( u _h , v _h ) _K + \langle \gamma u _h , v _h \rangle _{ \partial K } ,\\
  b _K ( v _h , q _h ) &\coloneqq \langle q _h , v _h \rangle _{ \partial K } .
\end{align*}
To separate the influence of $ \hat{u} _h $ and $f$, we define two
local solvers: Find
$ ( \mathsf{U} \hat{u} _h , \mathsf{P} \hat{u} _h ) \in V _h \times Q
_h $ and $ ( \mathsf{U} f , \mathsf{P} f ) \in V _h \times Q _h $ such
that
\begin{subequations}
  \label{e:local_solvers}
  \begin{align}
    a _K ( \mathsf{U} \hat{u} _h , v _h ) + b _K ( v _h , \mathsf{P} \hat{u} _h ) &= \langle \gamma \hat{u} _h , v _h \rangle _{ \partial K } , & a _K ( \mathsf{U} f, v _h ) + b _K ( v _h , \mathsf{P} f ) &= ( f, v _h ) _K , \label{e:local_solvers_v}\\
    b _K ( \mathsf{U} \hat{u} _h , q _h ) &= \langle \hat{u} _h , q _h \rangle _{ \partial K } , &   b _K ( \mathsf{U} f , q _h ) &= 0 \label{e:local_solvers_q},
  \end{align}
\end{subequations}
for all $ K \in \mathcal{T} _h $ and
$ ( v _h , q _h ) \in V _h \times Q _h $.

\begin{lemma}
  \label{l:locally_well-posed}
  If $ \alpha \geq 0 $, then the local solvers are well-defined, i.e.,
  \eqref{e:local_solvers} is uniquely solvable.
\end{lemma}

\begin{proof}
  First, we show that
  $ \sum _{ K \in \mathcal{T} _h } a _K ( \cdot , \cdot ) $ is
  coercive. This is obvious when $ \alpha > 0 $; when $ \alpha = 0 $,
  $ a _K ( u _h , u _h ) = 0 $ implies that
  $ u \rvert _{ \partial K } = 0 $, so Friedrichs's inequality implies
  $ u \rvert _K = 0 $. Finally, the surjectivity of
  $ v _h \mapsto \sum _{ K \in \mathcal{T} _h } b _K ( v _h , \cdot )
  = B _h v _h $ has already been shown in the proof of
  \cref{t:well-posed}.
\end{proof}

Assuming the local solvers are well-defined---which always holds for
$ \alpha \geq 0 $, by \cref{l:locally_well-posed}---we now define
$ \hat{\mathsf{P}} \hat{u} _h \coloneqq \mathsf{P} \hat{u} _h + \gamma
( \mathsf{U} \hat{u} _h - \hat{u} _h ) $ and
$ \hat{\mathsf{P}} f \coloneqq \mathsf{P} f + \gamma \mathsf{U} f $.
Substituting into \eqref{e:hybrid_vhat} and rearranging gives the
condensed problem: Find $ \hat{u} _h \in \hat{V} _h $ such that, for
all $ \hat{v} _h \in \hat{V} _h $,
\begin{equation}
  \label{e:condensed}
  - \langle \hat{\mathsf{P}} \hat{u} _h, \hat{v} _h \rangle _{ \partial \mathcal{T} _h } = \langle \hat{\mathsf{P}} f , \hat{v} _h \rangle _{ \partial \mathcal{T} _h } .
\end{equation}
Since the local solvers may be computed element-by-element, in
parallel if desired, the condensation from \eqref{e:hybrid} to
\eqref{e:condensed} is efficient to implement. The condensed bilinear
form
$ \hat{a} _h ( \hat{u} _h , \hat{v} _h ) \coloneqq - \langle
\hat{\mathsf{P}} \hat{u} _h, \hat{v} _h \rangle _{ \partial
  \mathcal{T} _h } $ on the left-hand side of \eqref{e:condensed} has
the following useful symmetric expression.

\begin{lemma}
  \label{l:condensed_bilinear}
  For all $ \hat{u} _h , \hat{v} _h \in \hat{V} _h $,
  \begin{multline*}
    \hat{a} _h ( \hat{u} _h , \hat{v} _h ) = ( \nabla \cdot \mathsf{U} \hat{u} _h , \nabla \cdot \mathsf{U} \hat{v} _h ) _{ \mathcal{T} _h } + ( \nabla \times \mathsf{U} \hat{u} _h , \nabla \times \mathsf{U} \hat{v} _h ) _{ \mathcal{T} _h } + \alpha ( \mathsf{U} \hat{u} _h , \mathsf{U} \hat{v} _h ) _{ \mathcal{T} _h } \\
    + \bigl\langle \gamma ( \mathsf{U} \hat{u} _h - \hat{u} _h ), \mathsf{U} \hat{v} _h - \hat{v} _h \bigr\rangle _{ \partial \mathcal{T} _h } .
  \end{multline*}
\end{lemma}

\begin{proof}
  We begin by writing
  \begin{equation*}
    - \langle \hat{\mathsf{P}} \hat{u} _h, \hat{v} _h \rangle _{ \partial \mathcal{T} _h } = - \langle \hat{\mathsf{P}} \hat{u} _h, \mathsf{U} \hat{v} _h \rangle _{ \partial \mathcal{T} _h } + \langle \hat{\mathsf{P}} \hat{u} _h, \mathsf{U} \hat{v} _h - \hat{v} _h \rangle _{ \partial \mathcal{T} _h } .
  \end{equation*}
  For the first term, \eqref{e:local_solvers_v} with
  $ v _h = \mathsf{U} \hat{v} _h $ implies
  \begin{equation*}
    - \langle \hat{\mathsf{P}} \hat{u} _h, \mathsf{U} \hat{v} _h \rangle _{ \partial \mathcal{T} _h } = ( \nabla \cdot \mathsf{U} \hat{u} _h , \nabla \cdot \mathsf{U} \hat{v} _h ) _{ \mathcal{T} _h } + ( \nabla \times \mathsf{U} \hat{u} _h , \nabla \times \mathsf{U} \hat{v} _h ) _{ \mathcal{T} _h } + \alpha ( \mathsf{U} \hat{u} _h , \mathsf{U} \hat{v} _h ) _{ \mathcal{T} _h } .
  \end{equation*}
  For the second term, \eqref{e:local_solvers_q} implies
  $ \langle \mathsf{P} \hat{u} _h , \mathsf{U} \hat{v} _h - \hat{v} _h
  \rangle _{ \partial \mathcal{T} _h } = 0 $, so
  \begin{equation*}
    \langle \hat{\mathsf{P}} \hat{u} _h, \mathsf{U} \hat{v} _h - \hat{v} _h \rangle _{ \partial \mathcal{T} _h } = \bigl\langle \gamma ( \mathsf{U} \hat{u} _h - \hat{u} _h ), \mathsf{U} \hat{v} _h - \hat{v} _h \bigr\rangle _{ \partial \mathcal{T} _h } ,
  \end{equation*}
  which completes the proof.
\end{proof}

\begin{theorem}
  \label{t:condensed_well-posed}
  Assuming the local solvers are well-defined,
  $ ( u _h , p _h , \hat{u} _h ) \in V _h \times Q _h \times \hat{V}
  _h $ is a solution of \eqref{e:hybrid} if and only if $ \hat{u} _h $
  is a solution of \eqref{e:condensed} with
  $ u _h = \mathsf{U} \hat{u} _h + \mathsf{U} f $ and
  $ p _h = \mathsf{P} \hat{u} _h + \mathsf{P} f $. Consequently,
  \eqref{e:hybrid} is uniquely solvable if and only if
  \eqref{e:condensed} is. In particular,
  $ \hat{a} _h ( \cdot , \cdot ) $ is symmetric positive-definite if
  $ \alpha > 0 $, and if the complement of $\Omega$ is connected, then
  this is also true for $ \alpha = 0 $.
\end{theorem}

\begin{proof}
  The equivalence of \eqref{e:hybrid} and \eqref{e:condensed} has
  already been demonstrated in the discussion above. When
  $ \alpha \geq 0 $, \cref{l:locally_well-posed} states that the local
  solvers are well-defined, and \cref{l:condensed_bilinear} implies
  that $ \hat{a} _h ( \cdot , \cdot ) $ is symmetric
  positive-semidefinite. Furthermore, if
  $ \hat{a} _h ( \hat{u} _h , \hat{u} _h ) = 0 $, then
  $ \mathsf{U} \hat{u} _h = \hat{u} _h $ on $ \mathcal{E} _h $, so
  $ \mathsf{U} \hat{u} _h \in H ( \operatorname{div} ; \Omega ) \cap
  \mathring{ H } ( \operatorname{curl} ; \Omega ) $ with
  $ a ( \mathsf{U} \hat{u} _h , \mathsf{U} \hat{u} _h ) = 0 $. Hence,
  as in the proof of \cref{t:well-posed},
  $ \hat{a} _h ( \cdot , \cdot ) $ is positive-definite whenever
  $ a ( \cdot , \cdot ) $ is.
\end{proof}

\begin{remark}
  These results tell us that static condensation from \eqref{e:hybrid}
  to \eqref{e:condensed} does not merely reduce the size of the global
  system. It also makes the system more amenable to efficient global
  solvers, such as the conjugate gradient method in the case where
  \eqref{e:condensed} is positive-definite.
\end{remark}

\section{Regularity and error analysis}
\label{s:analysis}

\subsection{Weighted Sobolev spaces and regularity}
\label{s:weighted}

\citet{CoDa2002} characterize the regularity of solutions to Maxwell's
equations in two dimensions (as well as in three) using a family of
weighted Sobolev spaces due to \citet{Kondratiev1967}. We now recall
these spaces and give corresponding regularity results for the problem
\eqref{e:primal}, combining the approach used in \citep{CoDa2002} with
that of \citet[Section 2]{BrCuLiSu2008}. For detailed treatments of
Kondrat'ev spaces and elliptic regularity in domains with corners, we
refer the reader to \citet*{NaPl1994} and \citet*{KoMaRo1997}.

As in \cref{s:description}, let $ r _\ell (x) $ denote the distance
from $ x \in \Omega $ to a corner $ c _\ell $ and
$ r ^\lambda \coloneqq \prod _{ \ell = 1 } ^L r _\ell ^{ \lambda _\ell
} $ for a multi-exponent
$ \lambda = ( \lambda _1, \ldots, \lambda _L ) $. Given a nonnegative
integer $m$, define the weighted Sobolev space
\begin{equation*}
  V ^m _\lambda (\Omega) \coloneqq \bigl\{ \phi \in \mathcal{D} ^\prime (\Omega) : r ^{ \lambda - m + \lvert \beta \rvert } \partial ^\beta \phi \in L ^2 (\Omega) \text{ for all } \lvert \beta \rvert \leq m \bigr\} ,
\end{equation*}
where $\beta$ is a multi-index, equipped with the natural norm defined
by
\begin{equation*}
  \lVert \phi \rVert ^2 _{ m, \lambda } \coloneqq \sum _{ \lvert \beta \rvert \leq m } \lVert r ^{ \lambda - m + \lvert \beta \rvert } \partial ^\beta \phi \rVert ^2 _\Omega .
\end{equation*}
This space also has the following equivalent characterization: If
$ \Omega = \Omega _0 \cup \bigcup _{ \ell = 1 } ^L \Omega _\ell $,
where $ \overline{ \Omega } _0 $ contains none of the corners and
$ \overline{ \Omega } _\ell $ contains only corner $ c_\ell $, then
\begin{equation*}
  V ^m _\lambda (\Omega) = \bigl\{ \phi \in \mathcal{D} ^\prime (\Omega) : \phi \rvert _{ \Omega _0 } \in H ^m (\Omega _0 ) \text{, and } r _\ell ^{ \lambda _\ell - m + \lvert \beta \rvert } \partial ^\beta \phi \rvert _{ \Omega _\ell } \in L ^2 (\Omega _\ell) \text{ for all } \ell \text{ and } \lvert \beta \rvert \leq m \bigr\} ,
\end{equation*}
since $ r _\ell \sim 1 $ on $ \Omega _0 $ for all $ \ell $, and
$ r _{\ell^\prime} \sim 1 $ on $ \Omega _\ell $ for $ \ell ^\prime \neq \ell $.

From the definitions, we immediately obtain the continuous inclusion
\begin{equation*}
  V ^{m + 1} _{\lambda + 1} (\Omega) \subset V ^m _\lambda (\Omega) ,
\end{equation*}
which may be interpolated to obtain fractional-order spaces. That is,
if $ s \geq 0 $, then $ V ^s _\lambda (\Omega) $ may be defined by
complex interpolation between
$ V ^{\lfloor s \rfloor + 1 } _{ \lambda - s + \lfloor s \rfloor + 1 }
(\Omega) $ and
$ V ^{\lfloor s \rfloor} _{ \lambda - s + \lfloor s \rfloor } (\Omega)
$, cf.~\citep[Section 2.4.5]{Schneider2021}. It follows that, more
generally,
\begin{equation*}
  V ^{ s + \epsilon } _{ \lambda + \epsilon } (\Omega) \subset V ^s _\lambda (\Omega) ,
\end{equation*}
for all $ s \geq 0 $ and $ \epsilon > 0 $. Additionally, the
continuous inclusions
$ V ^m _{\lambda ^\prime} (\Omega) \subset V ^m _\lambda (\Omega) $
for $ \lambda ^\prime \leq \lambda $ and
$ V ^m _0 (\Omega) \subset H ^m (\Omega) $ extend in the obvious way
from nonnegative integer $m$ to real $ s \geq 0 $.

\begin{remark}
  \citet{Schneider2021} uses an alternative notation for Kondrat'ev
  spaces,
  \begin{equation*}
    \mathcal{K} _{ p, a } ^m (\Omega) \coloneqq \bigl\{ \phi \in \mathcal{D} ^\prime (\Omega) : r ^{ \lvert \beta \rvert - a } \partial ^\beta \phi \in L ^p (\Omega) \text{ for all } \lvert \beta \rvert \leq m \bigr\} ,
  \end{equation*}
  so
  $ V _\lambda ^m (\Omega) = \mathcal{K} _{ 2, m - \lambda } ^m
  (\Omega) $. For example, the inclusion
  $ \mathcal{K} _{ p, a } ^{m+1} (\Omega) \subset \mathcal{K} _{ p, a
  } ^m (\Omega) $ in the notation of \citep{Schneider2021} gives
  $ V ^{m+1} _{\lambda + 1} (\Omega) \subset V ^m _\lambda (\Omega) $
  in our notation, since $ p = 2 $ and
  $ a = m - \lambda = ( m + 1 ) - ( \lambda + 1 ) $. Fractional
  Kondrat'ev spaces are denoted in \citep{Schneider2021} by
  $ \mathfrak{K} ^s _{ p , a } (\Omega) $, and similarly we have
  $ V ^s _\lambda (\Omega) = \mathfrak{K} _{ 2 , s - \lambda } ^s
  (\Omega) $.

  Finally, we note that an intrinsic treatment of fractional weighted
  Sobolev spaces may be found in \citet[Appendix A]{Dauge1988}.
\end{remark}

Suppose now that
$ u \in H ( \operatorname{div} ; \Omega ) \cap \mathring{ H } (
\operatorname{curl} ; \Omega ) $ satisfies \eqref{e:primal}. We recall
that $ \nabla \cdot u \in \mathring{ H } ^1 (\Omega) $,
since it can be seen as the solution to the Dirichlet problem
\begin{equation*}
  - \Delta (\nabla \cdot u) = \nabla \cdot ( f - \alpha u ) \in H ^{-1} (\Omega) .
\end{equation*}
Likewise, $ \nabla \times u \in H ^1 (\Omega) $, since it can be seen
as the zero-mean solution to a Neumann problem. See~\citet[Theorem
1.2]{CoDa2000} and similar arguments in \citet[Section
2]{BrCuLiSu2008}. Using this, we may now obtain a minimum weighted
Sobolev regularity result for $u$ itself.

\begin{theorem}
  \label{t:u_reg}
  If $u$ satisfies \eqref{e:primal}, then
  $ u \in \bigl[ V ^2 _{ 2 - 2 \mu + \epsilon } (\Omega) \bigr] ^2 $
  for all $ \epsilon > 0 $. Furthermore, if \eqref{e:primal} is
  well-posed, then we have the stability estimate
  $ \lVert u \rVert _{ 2 , 2 - 2 \mu + \epsilon } \lesssim \lVert f
  \rVert _\Omega $.
\end{theorem}

\begin{proof}
  As in \citet[Section 2]{BrCuLiSu2008}, it is sufficient to establish
  regularity and stability for $\Omega$ simply connected, since the
  general case follows by a partition of unity argument.

  Assuming $\Omega$ is simply connected, we can express $u$ in terms
  of its Helmholtz decomposition
  $ u = \nabla \phi + \nabla \times \psi $, where
  $ \phi \in \mathring{ H } ^1 (\Omega) $ and
  $ \psi \in H ^1 (\Omega) $ solve
  \begin{equation*}
    - \Delta \phi = - \nabla \cdot u , \qquad - \Delta \psi = \nabla \times u ,
  \end{equation*}
  with homogeneous Dirichlet and Neumann boundary conditions,
  respectively. For uniqueness, we take $ \int _\Omega \psi = 0 $. To
  determine the regularity of $\phi$ and $\psi$, we follow Chapter~2
  of \citet{NaPl1994}, which characterizes the regularity of solutions
  to Dirichlet and Neumann problems in plane domains with corner
  points; similar results are also found in
  \citet[\S6.6.1--6.6.2]{KoMaRo1997}.

  Hardy's inequality gives
  $ H ^1 (\Omega) \subset V ^1 _\epsilon (\Omega) $ for all
  $ \epsilon > 0 $, and $ \mu \leq 1 $ implies
  $ V ^1 _\epsilon (\Omega) \subset V ^1 _{ 2 - 2 \mu + \epsilon }
  (\Omega) $. Therefore, the right-hand sides $ - \nabla \cdot u $ and
  $ \nabla \times u $ are both in
  $ V ^1 _{ 2 - 2 \mu + \epsilon } (\Omega) $. Furthermore, since
  $ 0 < 2 \mu _\ell - \epsilon < \pi / \omega_\ell $ holds for all
  $ \ell $ when $\epsilon$ is sufficiently small, Theorem~3.1 in
  \citep[Chapter 2]{NaPl1994} implies
  \begin{align}
    \phi &\in V ^3 _{ 2 - 2 \mu + \epsilon } (\Omega), &\lVert \phi \rVert _{ 3, 2 - 2 \mu + \epsilon } &\lesssim \lVert \nabla \cdot u \rVert _{ 1 , 2 - 2 \mu + \epsilon },\notag\\
    \intertext{and Theorem~4.2 in \citep[Chapter 2]{NaPl1994} implies}
    \psi &\in V ^3 _{ 2 - 2 \mu + \epsilon } (\Omega), &\lVert \psi \rVert _{ 3, 2 - 2 \mu + \epsilon } &\lesssim \lVert \nabla \times u \rVert _{ 1 , 2 - 2 \mu + \epsilon }\notag.\\
    \intertext{By $ u = \nabla \phi + \nabla \times \psi $ and the continuity of the inclusion $ H ^1 (\Omega) \subset V ^1 _{ 2 - 2 \mu + \epsilon } (\Omega) $, we thus obtain}
    u &\in \bigl[ V ^2 _{ 2 - 2 \mu + \epsilon } (\Omega) \bigr] ^2 , & \lVert u \rVert _{ 2 , 2 - 2 \mu + \epsilon } &\lesssim \lVert \nabla \cdot u \rVert _1 + \lVert \nabla \times u \rVert _1 \label{e:weighted_stability},
  \end{align}
  where $ \lVert \cdot \rVert _1 $ is the $ H ^1 $ norm. This proves
  the first statement.

  Now, observe that the strong form \eqref{e:strong_pde} rearranges to
  \begin{equation*}
    - \nabla \nabla \cdot u + \nabla \times \nabla \times u = f - \alpha u ,
  \end{equation*}
  where the left-hand side is an $ L ^2 $-orthogonal sum. Therefore,
  by the Pythagorean theorem,
  \begin{equation*}
    \lvert \nabla \cdot u \rvert _1 ^2 + \lvert \nabla \times u \rvert _1 ^2 = \lVert f - \alpha u \rVert _\Omega ^2 \quad \Longrightarrow \quad \lvert \nabla \cdot u \rvert _1 + \lvert \nabla \times u \rvert _1 \lesssim \lVert f \rVert _\Omega + \lVert u \rVert _\Omega ,
  \end{equation*}
  where $ \lvert \cdot \rvert _1 $ is the $ H ^1 $ seminorm. Combining
  this with \eqref{e:weighted_stability} gives
  \begin{equation*}
    \lVert u \rVert _{ 2, 2 - 2 \mu + \epsilon } \lesssim \lVert f \rVert _\Omega + \lVert u \rVert _\Omega + \lVert \nabla \cdot u \rVert _\Omega + \lVert \nabla \times u \rVert _\Omega .
  \end{equation*}
  Finally, if \eqref{e:primal} is well-posed, then
  $ \lVert u \rVert _\Omega + \lVert \nabla \cdot u \rVert _\Omega +
  \lVert \nabla \times u \rVert _\Omega \lesssim \lVert f \rVert
  _\Omega $, which completes the proof.
\end{proof}

Most of the subsequent error analysis will use the following corollary
of \cref{t:u_reg}.

\begin{corollary}
  \label{c:u_reg}
  If $ s < \mu _\ell $ for all $ \ell $, then
  $ u \in \bigl[ V ^{ s + 1 } _{ 1 - \mu } (\Omega) \bigr] ^2 $ and
  $ \nabla \cdot u , \nabla \times u \in V ^s _{ \mu -1 } (\Omega)
  $. Furthermore, if \eqref{e:primal} is well-posed, then we have the
  stability estimate
  \begin{equation*}
    \lVert u \rVert _{ s + 1 , 1 - \mu } + \lVert \nabla \cdot u
    \rVert _{ s , \mu -1 } + \lVert \nabla \times u \rVert _{ s , \mu -1
    } \lesssim \lVert f \rVert _\Omega .
  \end{equation*}
\end{corollary}

\begin{proof}
  Pick $ \epsilon > 0 $ such that $ s < \mu _\ell - \epsilon $ for all
  $ \ell $. Then this follows by the continuous inclusions
  \begin{equation*}
    V ^2 _{ 2 - 2 \mu + \epsilon } (\Omega) = V ^{ s + 1 + ( 1 - s ) }_{ 1 - \mu + ( 1 - \mu + \epsilon ) } (\Omega) \subset V ^{ s + 1 + ( 1 - s ) } _{ 1 - \mu + ( 1 - s ) } (\Omega) \subset V ^{ s + 1 } _{ 1 - \mu } (\Omega) 
  \end{equation*}
  and
  \begin{equation*}
    H ^1 (\Omega) \subset V ^1 _\epsilon (\Omega) = V ^{ s + ( 1 - s ) } _{ \mu -1 + ( 1 - \mu + \epsilon ) } (\Omega) \subset V ^{ s + ( 1 - s ) } _{ \mu - 1 + ( 1 - s ) } (\Omega) \subset V ^s _{ \mu -1 } (\Omega) ,
  \end{equation*}
  together with \cref{t:u_reg}.
\end{proof}

Finally, we note that this also implies the following well-known
unweighted Sobolev regularity result, cf.~\citet*{AsCiSo1998}.

\begin{corollary}
  \label{c:u_reg_unweighted}
  If $ s < \min \bigl( 1, \pi / (2 \omega_\ell) \bigr) $ for all
  $ \ell $, then
  $ u \in \bigl[ V _0 ^{2 s} (\Omega) \bigr] ^2 \subset \bigl[ H ^{2 s}
  (\Omega) \bigr] ^2 $. Furthermore, if \eqref{e:primal} is
  well-posed, then we have the stability estimate
  $ \lVert u \rVert _{2 s} \lesssim \lVert f \rVert _\Omega $.
\end{corollary}

\begin{proof}
  Since $ \mu _\ell $ may be taken arbitrarily close to
  $ \min \bigl( 1, \pi / (2 \omega_\ell) \bigr) $, choose
  $ \mu _\ell $ so that $ s < \mu_\ell $ for all $ \ell $. Then
  \cref{c:u_reg} implies
  $ u \in \bigl[ V ^{ s + 1 } _{ 1 - \mu } (\Omega) \bigr] ^2 $, and
  we have the continuous inclusions
  \begin{equation*}
    V ^{ s + 1 } _{ 1 - \mu } (\Omega) = V ^{ 2 s + ( 1 - s ) } _{ 1 - \mu } (\Omega) \subset V ^{ 2 s + ( 1 - s ) } _{ 1 - s } (\Omega) \subset V _0 ^{ 2 s } (\Omega) \subset H ^{ 2 s } (\Omega) ,
  \end{equation*}
  which completes the proof.
\end{proof}

In particular, since $ \omega_\ell < 2 \pi $ for all $ \ell $, we may
take $ s > \frac{ 1 }{ 4 } $ in \cref{c:u_reg_unweighted} to conclude
that $ u \in \bigl[ H ^\sigma (\Omega) \bigr] ^2 $ with
$ \sigma > \frac{1}{2} $.

\subsection{Preliminary estimates}

We now establish two weighted Sobolev norm approximation results that
will be useful in the subsequent error analysis; compare Lemmas 5.2
and 5.3 in \citep{BrLiSu2007}.

For the remainder of the paper, we assume that $ \mathcal{T} _h $ is
shape-regular, but we make no additional assumptions about
quasi-uniformity or grading. Let $ h _K $ denote the diameter of
$ K \in \mathcal{T} _h $ and
$ h \coloneqq \max _{ K \in \mathcal{T} _h } h _K $. We denote the
weighted Sobolev norm on $ V ^s _\lambda (\Omega) \rvert _K $ by
$ \lVert \cdot \rVert _{ s, \lambda , K } $ (with distances taken to
the corners of $\Omega$, not those of $K$) and the ordinary Sobolev
seminorm on $ H ^s (K) $ by $ \lvert \cdot \rvert _{s, K } $.

\begin{lemma}
  \label{l:penalty_approx}
  If
  $ v \in \bigl[ H ^\sigma (\Omega) \cap V ^{ s + 1 } _{ 1 - \mu }
  (\Omega) \bigr] ^2 $ with $ \sigma > \frac{1}{2} $ and $ s \leq k $,
  then
  \begin{equation*}
    \frac{ \bigl[ \Phi _\mu (e) \bigr] ^2 }{ \lvert e \rvert } \lVert v - \Pi _h v \rVert _e ^2 \lesssim h _K ^{ 2 s } \lVert v \rVert _{ s + 1 , 1 - \mu, K } ^2 ,
  \end{equation*}
  for all $ K \in \mathcal{T} _h $ and $ e \subset \partial K $.
\end{lemma}

\begin{proof}
  If $K$ does not have any of the corners $ c _\ell $ as a vertex,
  then $ v \rvert _K \in \bigl[ H ^{ s + 1 } (K) \bigr] ^2 $, so the trace
  inequality with scaling and Bramble--Hilbert lemma imply
  \begin{equation*}
    \lvert e \rvert ^{-1} \lVert v - \Pi _h v \rVert _e ^2 \lesssim h _K ^{ - 2 } \lVert v - \Pi _h v \rVert _K ^2 + \lvert v - \Pi _h v \rvert _{ 1, K } ^2 \lesssim h _K ^{ 2 s } \lvert v \rvert ^2 _{ s + 1 , K } .
  \end{equation*}
  By shape regularity, we have
  $ \Phi _\mu (e) = r ^{ 1 - \mu } (m _e) \sim r ^{ 1 - \mu } (x) $
  for all $ x \in K $, and therefore
  \begin{equation*}
    \frac{ \bigl[ \Phi _\mu (e) \bigr] ^2 }{ \lvert e \rvert } \lVert v - \Pi _h v \rVert _e ^2 \lesssim h _K ^{ 2 s } \bigl[ \Phi _\mu (e) \bigr] ^2 \lvert v \rvert ^2 _{ s + 1 , K } \lesssim h _K ^{ 2 s } \lVert v \rVert ^2 _{ s + 1 , 1 - \mu , K } .
  \end{equation*}
  On the other hand, if $K$ has $ c _\ell $ as a vertex, then the
  inclusions
  $ V ^{ s + 1 } _{ 1 - \mu _\ell } (K) \subset V _0 ^{ s + \mu _\ell}
  (K) \subset H ^{ s + \mu_\ell } (K) $ imply
  $ v \rvert _K \in \bigl[ H ^{ s + \mu _\ell } (K) \bigr] ^2
  $. Hence, the trace inequality with scaling and Bramble--Hilbert
  give
  \begin{align*}
    \lvert e \rvert ^{-1} \lVert v - \Pi _h v \rVert _e ^2
    &\lesssim h _K ^{ - 2 } \lVert v - \Pi _h v \rVert _K ^2 + h _K ^{ 2 \min( 1, \sigma ) - 2 } \lvert v - \Pi _h v \rvert ^2 _{ \min (1, \sigma ), K }\\
    &\lesssim h _K ^{ 2 ( s + \mu _\ell ) - 2 } \lvert v \rvert _{ s + \mu_\ell, K } ^2 .
  \end{align*}  
  In this case,
  $ \Phi _\mu (e) \sim r _\ell ^{ 1 - \mu _\ell } ( m _e ) \sim h _K
  ^{ 1 - \mu _\ell } $, and therefore
  \begin{equation*}
    \frac{ \bigl[ \Phi _\mu (e) \bigr] ^2 }{ \lvert e \rvert } \lVert v - \Pi _h v \rVert _e ^2 \lesssim h _K ^{ 2 s } \lvert v \rvert ^2 _{ s + \mu _\ell , K } \lesssim h _K ^{ 2 s } \lVert v \rVert ^2 _{ s + 1 , 1 - \mu , K } ,
  \end{equation*}
  where the last inequality is due to the continuity of the inclusion
  $ V ^{ s + 1 } _{ 1 - \mu } (K) \subset H ^{ s + \mu _\ell } (K) $.
\end{proof}

\begin{lemma}
  \label{l:consistency_approx}
  If $ \eta \in H ^1 (\Omega) \cap V ^s _{ \mu -1 } (\Omega) $ with
  $ s \leq k $, then 
  \begin{equation*}
    \frac{ \lvert e \rvert }{ \bigl[ \Phi _\mu (e) \bigr] ^2 } \lVert \eta - P _h \eta \rVert _e ^2 \lesssim h _K ^{ 2 s } \lVert \eta \rVert _{ s , \mu -1, K } ^2 ,
  \end{equation*}
  for all $ K \in \mathcal{T} _h $ and $ e \subset \partial K $.
\end{lemma}

\begin{proof}
  Since $ \mu - 1 \leq 0 $, we have
  $ \eta \in V ^s _{ \mu -1 } (\Omega) \subset V _0 ^s (\Omega)
  \subset H ^s (\Omega) $. Thus, for all $ K \in \mathcal{T} _h $ and
  $ e \subset \partial K $, the trace inequality with scaling and
  Bramble--Hilbert lemma give
  \begin{equation*}
    \lvert e \rvert \lVert \eta - P _h \eta \rVert _e ^2 \lesssim \lVert \eta - P _h \eta \rVert _K ^2 + h _K ^2 \lvert \eta - P _h \eta \rvert _{1, K } ^2 \lesssim h _K ^{2 s } \lvert \eta \rvert ^2 _{ s, K } .
  \end{equation*}
  By shape regularity,
  $ \bigl[ \Phi _\mu (e) \bigr] ^{-1} = r ^{ \mu - 1 } ( m _e )
  \lesssim r ^{ \mu -1 } (x) $ for all $ x \in K $, and therefore
  \begin{equation*}
    \frac{ \lvert e \rvert }{ \bigl[ \Phi _\mu (e) \bigr] ^2 } \lVert \eta - P _h \eta \rVert _e ^2 \lesssim h _K ^{2 s} \bigl[ \Phi _\mu (e) \bigr] ^{-2} \lvert \eta \rvert ^2 _{ s, K } \lesssim h _K ^{2 s} \lVert \eta \rVert ^2 _{ s, \mu -1 , K } ,
  \end{equation*}
  which completes the proof.
\end{proof}

\subsection{Error estimates}
\label{s:error_estimates}

We now estimate the error $ u - u _h $, where $u$ satisfies
\eqref{e:primal} and $ u _h $ satisfies \eqref{e:one-field}. The
argument follows a similar general outline to that in
\citet{BrCuLiSu2008}, but the details differ in several important
respects---especially in the use of weighted Sobolev regularity
hypotheses and higher-order polynomial approximation, and in the
absence of mesh-grading assumptions.

As in \citep{BrCuLiSu2008}, we will first estimate the error in the
mesh-dependent energy norm
\begin{equation*}
  \lVert v \rVert _h ^2 \coloneqq \lVert v \rVert _\Omega ^2 + \lVert \nabla \cdot v \rVert _{ \mathcal{T} _h } ^2 + \lVert \nabla \times v \rVert _{ \mathcal{T} _h } ^2 + \frac{1}{2} \bigl\langle \gamma \llbracket v \rrbracket , \llbracket v \rrbracket \bigr\rangle _{ \mathcal{E} _h ^\circ } + \langle \gamma v \times n , v \times n \rangle _{ \mathcal{E} _h ^\partial } .
\end{equation*}
If we extend $ a _h ( \cdot , \cdot ) $ from $ \mathring{ V } _h $ to
$ H ( \operatorname{div}; \Omega ) \cap \mathring{ H } (
\operatorname{curl} ; \Omega ) + \mathring{ V } _h $, then in the
special case $ \alpha = 1 $, this is precisely the norm associated to
$ a _h ( \cdot , \cdot ) $ considered as an inner product.

For arbitrary $\alpha$, we immediately see that
$ a _h ( \cdot , \cdot ) $ is bounded with respect to
$ \lVert \cdot \rVert _h $. For $ \alpha > 0 $, we have the coercivity
condition
\begin{equation*}
  a _h ( v , v ) \geq \min ( 1, \alpha ) \lVert v \rVert _h ^2  .
\end{equation*}
If the complement of $\Omega$ is connected, then we also have
coercivity for $ \alpha = 0 $, by the argument in the proof of
\cref{t:well-posed}.  In general, for $ \alpha \leq 0 $, we have a
G\aa rding inequality (which is actually an equality),
\begin{equation*}
  a _h ( v, v ) + \bigl( \lvert \alpha \rvert + 1 \bigr) \lVert v \rVert ^2 _\Omega = \lVert v \rVert _h ^2 .
\end{equation*}
This implies the following Strang-type abstract estimates, whose
proofs are identical to those of Lemma~3.5 and Lemma~3.6 in
\citet{BrLiSu2008}.

\begin{lemma}
  \label{l:strang}
  If $ \alpha > 0 $, $u$ is the solution to \eqref{e:primal}, and
  $ u _h $ is the solution to \eqref{e:one-field}, then
  \begin{equation}
    \label{e:strang_coercive}
    \lVert u - u _h \rVert _h \lesssim \inf _{ v _h \in \mathring{ V } _h } \lVert u - v _h \rVert _h + \sup _{ 0 \neq w _h \in \mathring{ V } _h } \frac{ a _h ( u - u _h , w _h ) }{ \lVert w _h \rVert _h } ,
  \end{equation}
  and if the complement of $\Omega$ is connected, then this also holds
  for $ \alpha = 0 $. If $ \alpha \leq 0 $, $u$ satisfies
  \eqref{e:primal}, and $ u _h $ satisfies \eqref{e:one-field}, then
  \begin{equation}
    \label{e:strang_garding}
    \lVert u - u _h \rVert _h \lesssim \inf _{ v _h \in \mathring{ V } _h } \lVert u - v _h \rVert _h + \sup _{ 0 \neq w _h \in \mathring{ V } _h } \frac{ a _h ( u - u _h , w _h ) }{ \lVert w _h \rVert _h } + \lVert u - u _h \rVert _\Omega .
  \end{equation}
\end{lemma}

We will proceed by estimating the two terms on the right-hand side of
\eqref{e:strang_coercive}, which correspond to approximation error and
consistency error, respectively.

\begin{lemma}
  \label{l:energy_approx}
  If $u \in \bigl[ V ^{ s + 1 } _{ 1 - \mu } (\Omega) \bigr] ^2 $ and
  $ \nabla \cdot u , \nabla \times u \in V ^s _{ \mu -1 } (\Omega) $
  with $ s \leq k $, then
  \begin{equation*}
    \inf _{ v _h \in \mathring{ V } _h } \lVert u - v _h \rVert _h \leq \lVert u - \Pi _h u \rVert _h \lesssim h ^s \bigl( \lVert u \rVert _{ s + 1, 1 - \mu } + \lVert \nabla \cdot u \rVert _{ s, \mu -1 } + \lVert \nabla \times u \rVert _{ s , \mu -1 } \bigr) .
  \end{equation*}
\end{lemma}

\begin{proof}
  The first inequality holds since the BSM projection maps
  $ u \in \bigl[ H ^\sigma (\Omega) \bigr] ^2 $ with
  $ \sigma > \frac{1}{2} $ to $ \Pi _h u \in \mathring{ V } _h
  $. Next, letting $ \mu_{\min} \coloneqq \min_\ell \mu _\ell $, the
  continuous inclusion
  $ V ^{ s + 1 } _{ 1 - \mu } (\Omega) \subset H ^{ s + \mu_{\min}}
  (\Omega) $ implies that
  $ u \in \bigl[ H ^{ s + \mu_{\min} } (\Omega) \bigr] ^2 $, so
  polynomial approximation theory gives
  \begin{equation}
    \label{e:l2_approx}
    \lVert u - \Pi _h u \rVert _K ^2 \lesssim h _K ^{2 (s + \mu _{\min} ) }  \lvert u \rvert _{ s + \mu_{\min} , K } ^2 \lesssim h _K ^{2 (s + \mu _{\min} ) }  \lVert  u \rVert _{ s + 1, 1 - \mu , K } ^2 .
  \end{equation}
  Furthermore, \cref{l:commuting_projection} and
  $ V ^s _{ \mu -1 } (\Omega) \subset H ^s (\Omega) $ imply
  \begin{subequations}
    \label{e:div_curl_approx}
    \begin{alignat}{3}
      \label{e:div_approx}
      \bigl\lVert \nabla \cdot ( u - \Pi _h u ) \bigr\rVert _K ^2 &= \bigl\lVert ( \nabla \cdot u ) - P _h ( \nabla \cdot u ) \bigr\rVert _K ^2 &&\lesssim h _K ^{2 s } \lvert \nabla \cdot u \rvert _{ s, K } ^2 &&\lesssim h _K ^{2 s } \lVert  \nabla \cdot u \rVert _{ s, \mu -1 , K } ^2, \\
      \label{e:curl_approx}
      \bigl\lVert \nabla \times ( u - \Pi _h u ) \bigr\rVert _K ^2 &= \bigl\lVert ( \nabla \times u ) - P _h ( \nabla \times u ) \bigr\rVert _K ^2 &&\lesssim h _K ^{2 s } \lvert \nabla \times u \rvert _{ s, K } ^2 &&\lesssim h _K ^{2 s } \lVert  \nabla \times u \rVert _{ s, \mu -1 , K } ^2  .
    \end{alignat}
  \end{subequations}
  It remains to estimate the contributions from the penalty
  terms. Observe that, for $ e \in \mathcal{E} _h ^\circ $,
  \begin{equation*}
    \bigl\langle \gamma \llbracket u - \Pi _h u \rrbracket , \llbracket u - \Pi _h u \rrbracket \bigr\rangle _e \leq 2 \Bigl( \bigl\langle \gamma ( u - \Pi _h u ) , u - \Pi _h u \bigr\rangle _{ e ^+ } + \bigl\langle \gamma ( u - \Pi _h u ) , u - \Pi _h u \bigr\rangle _{ e ^- } \Bigr),
  \end{equation*}
  by the parallelogram identity. Therefore,
  \begin{equation*}
    \bigl\langle \gamma \llbracket u - \Pi _h u \rrbracket , \llbracket u - \Pi _h u \rrbracket \bigr\rangle _{ \mathcal{E} _h ^\circ } + \bigl\langle \gamma ( u - \Pi _h u ) \times n , ( u - \Pi _h u ) \times n \bigr\rangle _{ \mathcal{E} _h ^\partial } \leq 2 \bigl\langle \gamma ( u - \Pi _h u ) , u - \Pi _h u \bigr\rangle _{ \partial \mathcal{T} _h } ,
  \end{equation*}
  so it suffices to estimate the contribution from each
  $ K \in \mathcal{T} _h $ and $ e \subset \partial K $. By
  \cref{l:penalty_approx},
  \begin{equation}
    \label{e:penalty_approx}
    \bigl\langle \gamma ( u - \Pi _h u ) , ( u - \Pi _h u ) \bigr\rangle _{ \partial K } \lesssim h _K ^{ 2 s } \lVert u \rVert _{ s + 1, 1 - \mu , K } ^2 .
  \end{equation}
  Finally, combining \eqref{e:l2_approx}, \eqref{e:div_curl_approx},
  and \eqref{e:penalty_approx} and summing over
  $ K \in \mathcal{T} _h $ completes the proof.  
\end{proof}

\begin{lemma}
  \label{l:consistency}
  Suppose $u$ satisfies \eqref{e:primal} and $ u _h $ satisfies
  \eqref{e:one-field}. If
  $ \nabla \cdot u , \nabla \times u \in V ^s _{ \mu -1 } (\Omega) $
  with $ s \leq k $, then
  \begin{equation*}
    \sup _{ 0 \neq w _h \in \mathring{ V } _h } \frac{ a _h ( u - u _h , w _h ) }{ \lVert w _h \rVert _h } \lesssim h ^s \bigl( \lVert \nabla \cdot u \rVert _{ s, \mu -1 } + \lVert \nabla \times u \rVert _{ s , \mu -1 } \bigr) .
  \end{equation*}  
\end{lemma}

\begin{proof}
  Subtracting \eqref{e:one-field} from \eqref{e:ibp} with
  $ v = v _h = w _h \in \mathring{ V } _h $, we get
  \begin{align}
    a _h ( u - u _h , w _h )
    &= \langle \nabla \cdot u , w _h \cdot n \rangle _{ \partial \mathcal{T} _h } - \langle \nabla \times u , w _h \times n \rangle _{ \partial \mathcal{T} _h } \notag \\
    &= \bigl\langle \nabla \cdot u , \llbracket w _h \cdot n \rrbracket \bigr\rangle _{ \mathcal{E} _h ^\circ } - \bigl\langle \nabla \times u , \llbracket w _h \times n \rrbracket \bigr\rangle _{ \mathcal{E} _h ^\circ } - \langle \nabla \times u , w _h \times n \rangle _{ \mathcal{E} _h ^\partial } \label{e:consistency_jump},
  \end{align}
  where we have denoted the normal and tangential jump components on
  $ e \in \mathcal{E} _h ^\circ $ by
  \begin{equation*}
    \llbracket w _h \cdot n \rrbracket _e \coloneqq w _h ^+ \cdot n ^+ + w _h ^- \cdot n ^- , \qquad \llbracket w _h \times n \rrbracket _e \coloneqq w _h ^+ \times n ^+ + w _h ^- \times n ^- .
  \end{equation*}
  The condition $ w _h \in \mathring{ V } _h $ says that
  $ \llbracket w _h \cdot n \rrbracket _e $ and
  $ \llbracket w _h \times n \rrbracket _e $ are each
  $ L ^2 $-orthogonal to $ P _{ k -1 } (e) $ for
  $ e \in \mathcal{E} _h ^\circ $, and that $ w _h \times n \rvert _e $ is
  $ L ^2 $-orthogonal to $ P _{ k -1 } (e) $ for
  $ e \in \mathcal{E} _h ^\partial $. Therefore, letting $ P _h $ be
  projection onto either triangle $ K ^\pm $ containing
  $ e \in \mathcal{E} _h ^\circ $, since
  $ P _h ( \nabla \cdot u ) \rvert _e \in P _{ k -1 } (e) $, it
  follows that
  \begin{align*}
    \bigl\langle \nabla \cdot u , \llbracket w _h \cdot n \rrbracket \bigr\rangle _{ \mathcal{E} _h ^\circ }  
    &= \bigl\langle \nabla \cdot u - P _h ( \nabla \cdot u ) , \llbracket w _h \cdot n \rrbracket \bigr\rangle _{ \mathcal{E}_h^\circ } \\
    &= \Bigl\langle \gamma ^{ - 1/2 } \bigl[ \nabla \cdot u - P _h ( \nabla \cdot u ) \bigr] , \gamma ^{ 1/2 } \llbracket w _h \cdot n \rrbracket \Bigr\rangle _{ \mathcal{E}_h^\circ } \\
    &\leq \Bigl\lVert \gamma ^{ - 1/2 } \bigl[ \nabla \cdot u - P _h ( \nabla \cdot u ) \bigr] \Bigr\rVert _{ \mathcal{E} _h ^\circ } \Bigl\lVert \gamma ^{ 1/2 } \llbracket w _h \cdot n \rrbracket \Bigr\rVert _{ \mathcal{E} _h ^\circ } ,
  \end{align*}
  where the last step uses the Cauchy--Schwarz inequality. Applying
  \cref{l:consistency_approx} with $ \eta = \nabla \cdot u $ to the
  first term and the definition of the energy norm to the second, we
  conclude that
  \begin{equation*}
    \bigl\langle \nabla \cdot u , \llbracket w _h \cdot n \rrbracket \bigr\rangle _{ \mathcal{E} _h ^\circ } \lesssim h ^s \bigl\lVert \nabla \cdot u \rVert _{ s, \mu -1 } \lVert w _h \rVert _h .
  \end{equation*}
  Similarly,
  \begin{equation*}
- \bigl\langle \nabla \times u , \llbracket w _h \times n \rrbracket \bigr\rangle _{ \mathcal{E} _h ^\circ } - \langle \nabla \times u , w _h \times n \rangle _{ \mathcal{E} _h ^\partial } \lesssim h ^s \lVert \nabla \times u \rVert _{ s, \mu -1 } \lVert w _h \rVert _h ,    
  \end{equation*}
  and the result follows.
\end{proof}

Next, we use a duality argument to control the error in the $ L ^2 $
norm.

\begin{lemma}
  \label{l:duality}
  Suppose $u$ satisfies \eqref{e:primal} and $ u _h $ satisfies
  \eqref{e:one-field}. Suppose also that
  $ \nabla \cdot u , \nabla \times u \in V ^s _{ \mu -1 } (\Omega) $
  with $ s \leq k $, and let $ t < \mu_{\min} $. If \eqref{e:primal}
  is well-posed, then
  \begin{equation*}
    \lVert u - u _h \rVert _\Omega \lesssim h ^{ s + t } \bigl( \lVert \nabla \cdot u \rVert _{ s, \mu -1 } + \lVert \nabla \times u \rVert _{ s , \mu -1 } \bigr) + h ^t \lVert u - u _h \rVert _h 
  \end{equation*}
\end{lemma}

\begin{proof}
  Let
  $ z \in H ( \operatorname{div}; \Omega ) \cap \mathring{ H } (
  \operatorname{curl}; \Omega ) $ be the solution to
  \begin{equation}
    \label{e:dual}
    ( \nabla \cdot v, \nabla \cdot z ) _\Omega + ( \nabla \times v, \nabla \times z ) _\Omega + \alpha ( v, z ) _\Omega = ( v, u - u _h ) _\Omega ,
  \end{equation}
  for all
  $ v \in H ( \operatorname{div}; \Omega ) \cap \mathring{ H } (
  \operatorname{curl}; \Omega ) $. By \cref{c:u_reg},
  $ z \in \bigl[ V ^{ t + 1 } _{ 1 - \mu } (\Omega) \bigr] ^2 $ and
  $ \nabla \cdot z , \nabla \times z \in V ^t _{ \mu -1 } (\Omega) $,
  and we have the stability estimate
  \begin{equation}
    \label{e:z_stability}
    \lVert z \rVert _{ t + 1 , 1 - \mu } + \lVert \nabla \cdot z \rVert _{ t, \mu -1 } + \lVert \nabla \times z \rVert _{ t, \mu -1 } \lesssim \lVert u - u _h \rVert _\Omega .
  \end{equation}
  Hence, \cref{l:energy_approx} implies
  \begin{equation}
    \label{e:z_approx}
    \lVert z - \Pi _h z \rVert _h \lesssim h ^t \lVert u - u _h \rVert _\Omega .
  \end{equation}
  To express $ \lVert u - u _h \rVert _\Omega ^2 $ in terms of $z$, we
  would like to take $ v = u - u _h $ in \eqref{e:dual}, but we cannot
  do so since generally
  $ u _h \notin H ( \operatorname{div}; \Omega ) \cap \mathring{ H } (
  \operatorname{curl} ; \Omega ) $. Instead, integrating by parts as
  in \eqref{e:ibp} gives
  \begin{align}
    \lVert u - u _h \rVert _\Omega ^2
    &= a _h ( u - u _h , z ) + \langle u _h \cdot n , \nabla \cdot z \rangle _{ \partial \mathcal{T} _h } - \langle u _h \times n , \nabla \times z \rangle _{ \partial \mathcal{T} _h }, \notag\\
    &= a _h ( u - u _h , z ) + \bigl\langle \llbracket u _h \cdot n \rrbracket , \nabla \cdot z \bigr\rangle _{ \mathcal{E} _h ^\circ } - \bigl\langle \llbracket u _h \times n \rrbracket , \nabla \times z \bigr\rangle _{ \mathcal{E} _h ^\circ } - \langle u _h \times n , \nabla \times z \rangle _{ \mathcal{E} _h ^\partial }, \label{e:l2_error}
  \end{align}
  which we will estimate term-by-term.

  For the first term of \eqref{e:l2_error}, we write
  \begin{equation*}
    a _h ( u - u _h , z ) = a _h ( u - u _h , z - \Pi _h z ) + a _h ( u - u _h , \Pi _h z ) .
  \end{equation*}
  By the boundedness of $ a _h ( \cdot , \cdot ) $ in the energy norm
  and \eqref{e:z_approx}, we have
  \begin{equation*}
    a _h ( u - u _h , z - \Pi _h z ) \lesssim h ^t \lVert u - u _h \rVert _h \lVert u - u _h \rVert _\Omega .
  \end{equation*}
  Next, by \eqref{e:consistency_jump} with
  $ w _h = \Pi _h z \in \mathring{ V } _h $, we have
  \begin{equation*}
    a _h ( u - u _h , \Pi _h z ) = \bigl\langle \nabla \cdot u , \llbracket \Pi _h z \cdot n \rrbracket \bigr\rangle _{ \mathcal{E} _h ^\circ } - \bigl\langle \nabla \times u, \llbracket \Pi _h z \times n \rrbracket \bigr\rangle _{ \mathcal{E} _h ^\circ } - \langle \nabla \times u , \Pi _h z \times n \rangle _{ \mathcal{E} _h ^\partial } .
  \end{equation*}
  By a similar argument to that used in \cref{l:consistency}, along
  with the fact that $ \llbracket z \cdot n \rrbracket = 0 $,
  \begin{align*}
    \bigl\langle \nabla \cdot u , \llbracket \Pi _h z \cdot n \rrbracket \bigr\rangle _{ \mathcal{E} _h ^\circ }
    &= \Bigl\langle \nabla \cdot u - P _h ( \nabla \cdot u ) , \bigl\llbracket (\Pi _h z - z ) \cdot n \bigr\rrbracket \Bigr\rangle _{ \mathcal{E} _h ^\circ } \\
    &= \Bigl\langle \gamma ^{ -1/2 } \bigl[ \nabla \cdot u - P _h ( \nabla \cdot u ) \bigr] , \gamma ^{ 1/2 } \bigl\llbracket (\Pi _h z - z ) \cdot n \bigr\rrbracket  \Bigr\rangle _{ \mathcal{E} _h ^\circ } \\
    &\leq \Bigl\lVert \gamma ^{ -1/2 } \bigl[ \nabla \cdot u - P _h ( \nabla \cdot u ) \bigr] \Bigr\rVert _{ \mathcal{E} _h ^\circ } \Bigl\lVert \gamma ^{ 1/2 } \bigl\llbracket ( z - \Pi _h z ) \cdot n \bigr\rrbracket \Bigr\rVert _{ \mathcal{E} _h ^\circ } \\
    &\lesssim h ^s \lVert \nabla \cdot u \rVert _{ s, \mu -1 } \lVert z - \Pi _h z \rVert _h \\
    &\lesssim h ^{s + t} \lVert \nabla \cdot u \rVert _{ s, \mu -1 } \lVert u - u _h \rVert _\Omega ,
  \end{align*}
  where the last two lines use \cref{l:consistency_approx} with
  $ \eta = \nabla \cdot u $ and \eqref{e:z_approx}. Similarly,
  \begin{equation*}
    - \bigl\langle \nabla \times u, \llbracket \Pi _h z \times n \rrbracket \bigr\rangle _{ \mathcal{E} _h ^\circ } - \langle \nabla \times u , \Pi _h z \times n \rangle _{ \mathcal{E} _h ^\partial } \lesssim h ^{ s + t } \lVert \nabla \times u \rVert _{ s , \mu -1 } \lVert u - u _h \rVert _\Omega .
  \end{equation*}
  Thus, we have estimated the first term of \eqref{e:l2_error} by
  \begin{equation}
    \label{e:l2_error_term1}
    a _h ( u - u _h , z ) \lesssim \Bigl[ h ^{ s + t } \bigl( \lVert \nabla \cdot u \rVert _{ s, \mu -1 } + \lVert \nabla \times u \rVert _{ s, \mu -1 } \bigr) + h ^t \lVert u - u _h \rVert _h \Bigr] \lVert u - u _h \rVert _\Omega .
  \end{equation}
  For the remaining terms of \eqref{e:l2_error}, we use a similar
  argument to the one above to get
  \begin{align}
    \bigl\langle \llbracket u _h \cdot n \rrbracket , \nabla \cdot z \bigr\rangle _{ \mathcal{E} _h ^\circ }
    &= \Bigl\langle \bigl\llbracket (u _h - u ) \cdot n \bigr\rrbracket , \nabla \cdot z - P _h ( \nabla \cdot z ) \Bigr\rangle _{ \mathcal{E} _h ^\circ } \notag \\
    &= \Bigl\langle \gamma ^{ 1/2 } \bigl\llbracket ( u _h - u ) \cdot n \bigr\rrbracket, \gamma ^{ -1/2 } \bigl[ \nabla \cdot z - P _h ( \nabla \cdot z ) \bigr]  \Bigr\rangle _{ \mathcal{E} _h ^\circ } \notag \\
    &\leq \Bigl\lVert \gamma ^{ 1/2 } \bigl\llbracket (u - u _h) \cdot n \bigr\rrbracket \Bigr\rVert _{ \mathcal{E} _h ^\circ} \Bigl\lVert \gamma ^{ -1/2 } \bigl[ \nabla \cdot z - P _h ( \nabla \cdot z ) \bigr] \Bigr\rVert _{ \mathcal{E} _h ^\circ} \notag \\
    &\lesssim h ^t \lVert u - u _h \rVert _h \lVert \nabla \cdot z \rVert _{ t, \mu -1 } \notag \\
    &\lesssim h ^t \lVert u - u _h \rVert _h \lVert u - u _h \rVert _\Omega \label{e:l2_error_term2},
  \end{align}
  where the last two lines use \cref{l:consistency_approx} with
  $ \eta = \nabla \cdot z $ and \eqref{e:z_stability}. Likewise,
  \begin{equation}
    - \bigl\langle \llbracket u _h \times n \rrbracket , \nabla \times z \bigr\rangle _{ \mathcal{E} _h ^\circ } - \langle u _h \times n , \nabla \times z \rangle _{ \mathcal{E} _h ^\partial } \lesssim h ^t \lVert u - u _h \rVert _h \lVert \nabla \times z \rVert _{ t, \mu -1 } \lesssim h ^t \lVert u - u _h \rVert _h \lVert u - u _h \rVert _\Omega \label{e:l2_error_term3}.
  \end{equation}
  Altogether, estimating \eqref{e:l2_error} by combining
  \eqref{e:l2_error_term1}, \eqref{e:l2_error_term2}, and
  \eqref{e:l2_error_term3}, we have
  \begin{equation*}
    \lVert u - u _h \rVert _\Omega ^2  \lesssim \Bigl[ h ^{ s + t } \bigl( \lVert \nabla \cdot u \rVert _{ s, \mu -1 } + \lVert \nabla \times u \rVert _{ s, \mu -1 } \bigr) + h ^t \lVert u - u _h \rVert _h \Bigr] \lVert u - u _h \rVert _\Omega ,
  \end{equation*}
  which completes the proof.
\end{proof}

Finally, we are ready to state the main energy and $ L ^2 $ error
estimates.

\begin{theorem}
  \label{t:error_estimates}
  Suppose $u$ satisfies \eqref{e:primal} with
  $ u \in \bigl[ V ^{ s + 1 } _{ 1 - \mu } (\Omega) \bigr] ^2 $ and
  $ \nabla \cdot u, \nabla \times u \in V ^s _{ \mu -1 } (\Omega) $,
  where $ s \leq k $, and let $ t < \mu_{\min} $. If $ \alpha > 0 $,
  then the solution $ u _h $ to \eqref{e:one-field} satisfies the
  error estimates
  \begin{align*}
    \lVert u - u _h \rVert _h &\lesssim h ^s \bigl( \lVert u \rVert _{ s + 1 , 1 - \mu } + \lVert \nabla \cdot u \rVert _{ s, \mu -1 } + \lVert \nabla \times u \rVert _{ s, \mu -1 } \bigr) ,\\
    \lVert u - u _h \rVert _\Omega &\lesssim h ^{ s + t } \bigl( \lVert u \rVert _{ s + 1 , 1 - \mu } + \lVert \nabla \cdot u \rVert _{ s, \mu -1 } + \lVert \nabla \times u \rVert _{ s, \mu -1 } \bigr),
  \end{align*}
  and if the complement of $\Omega$ is connected, then this also holds
  for $ \alpha = 0 $. If $ \alpha < 0 $ and \eqref{e:primal} is
  well-posed, then \eqref{e:one-field} is uniquely solvable for
  sufficiently small $h$, and the solution $ u _h $ satisfies these
  same estimates.
\end{theorem}

\begin{proof}
  If $ \alpha > 0 $, or if $ \alpha = 0 $ with $\Omega$ having
  connected complement, then the proof is fairly immediate. The energy
  estimate follows from the abstract estimate
  \eqref{e:strang_coercive} in \cref{l:strang}, together with
  \cref{l:energy_approx,l:consistency}, and the $ L ^2 $ estimate
  follows by \cref{l:duality}.

  When $ \alpha < 0 $ is such that \eqref{e:primal} is well-posed, we
  follow the approach in \citet[Theorem 4.5]{BrLiSu2008}, which uses a
  technique for indefinite problems due to \citet{Schatz1974}. Suppose
  that $ u _h $ satisfies \eqref{e:one-field}.  From the abstract
  estimate \eqref{e:strang_garding} in \cref{l:strang}, along with
  \cref{l:energy_approx,l:consistency,,l:duality}, we have
  \begin{equation}
    \label{e:schatz}
    \lVert u - u _h \rVert _h \leq C \Bigl[ h ^s \bigl( \lVert u \rVert _{ s + 1 , 1 - \mu } + \lVert \nabla \cdot u \rVert _{ s, \mu -1 } + \lVert \nabla \times u \rVert _{ s, \mu -1 } \bigr) + h ^t \lVert u - u _h \rVert _h \Bigr] ,
  \end{equation}
  where the constant $C$ has been made explicit. Now, choose
  $ h _\ast $ small enough that $ C h _\ast ^t < 1 $. It follows that,
  whenever $ h \leq h _\ast $, we may subtract
  $ C h ^t \lVert u - u _h \rVert _h $ from both sides of
  \eqref{e:schatz} to obtain
  \begin{equation*}
    \lVert u - u _h \rVert _h \lesssim h ^s \bigl( \lVert u \rVert _{ s + 1 , 1 - \mu } + \lVert \nabla \cdot u \rVert _{ s , \mu -1 } + \lVert \nabla \times u \rVert _{ s , \mu -1 } \bigr) .
  \end{equation*}
  In particular, when $ f = 0 $, well-posedness of \eqref{e:primal}
  gives the unique solution $ u = 0 $, and
  $ \lVert u _h \rVert _h \lesssim 0 $ implies that
  \eqref{e:one-field} has the unique solution $ u _h = 0 $. Hence,
  \eqref{e:one-field} is uniquely solvable and satisfies the energy
  estimate whenever $ h \leq h _\ast $, and the $ L ^2 $ estimate
  follows by another application of \cref{l:duality}.
\end{proof}

\begin{corollary}[minimum-regularity case]
  If $ \alpha > 0 $, $u$ is the solution to \eqref{e:primal}, and
  $ u _h $ is the solution to \eqref{e:one-field}, then for all
  $ s < \mu_{\min} $, we have the error estimates
  \begin{align*}
    \lVert u - u _h \rVert _h &\lesssim h ^s \lVert f \rVert _\Omega ,\\
    \lVert u - u _h \rVert _\Omega &\lesssim h ^{ 2 s } \lVert f \rVert _\Omega ,
  \end{align*}
  and if the complement of $\Omega$ is connected, then this also holds
  for $ \alpha = 0 $. If $ \alpha < 0 $ and \eqref{e:primal} is
  well-posed, then these estimates hold for sufficiently small $h$.
\end{corollary}

\begin{proof}
  This is immediate from \cref{c:u_reg,t:error_estimates} with
  $ s = t $.
\end{proof}

\section{Numerical experiments}
\label{s:experiments}

In this section, we present numerical experiments illustrating the
convergence behavior of the method, showing how convergence is
affected by the interior angles of $\Omega$ and by the regularity of
the exact solution $u$, and relating these numerical results to the
theoretical results of \cref{s:analysis}. For all numerical
experiments, we take $ \alpha = 1 $.

All computations have been carried out using the Firedrake finite
element library \citep{RaHaMiLaLuMcBeMaKe2017} (version
0.13.0+4959.gac22e4c5), and a Firedrake component called Slate
\citep{GiMiHaCo2020} was used to implement the local solvers for
static condensation and postprocessing.

\subsection{Smooth solution on square domain}
\label{s:square}

We begin by considering the square domain
$ \Omega = \bigl( 0, \frac{1}{2} \bigr) ^2 $. Since all four corners
of $\Omega$ are $ \pi / 2 $, we have $ \mu = (1, 1, 1, 1) $. Given
$N \in \mathbb{N} $, we construct a uniform triangle mesh by
partitioning $\Omega$ uniformly into $ N \times N $ squares, then
dividing each into two triangles.

\begin{table}
  \centering
  \scriptsize
  \begin{filecontents*}{square.csv}
N,k=1_energy_error,k=1_energy_rate,k=1_l2_error,k=1_l2_rate,k=2_energy_error,k=2_energy_rate,k=2_l2_error,k=2_l2_rate,k=3_energy_error,k=3_energy_rate,k=3_l2_error,k=3_l2_rate
2,0.0026159682823922517,,0.000410282328185225,,0.0007214121192776369,,8.793429848078631e-05,,0.00015238957822650707,,1.6097970907263602e-05,
4,0.0013040112332617772,1.0043887512596434,8.977029610191723e-05,2.1923069565209783,0.00020206302025119896,1.8360183428541124,1.2724179703665011e-05,2.7888533412204097,2.0203428040997496e-05,2.9150922318297,9.99851662823838e-07,4.009020969222616
8,0.0006387089226594199,1.0297257875816825,2.039106968894469e-05,2.1383006944603196,5.3081146797732135e-05,1.9285338718316414,1.7412626480467507e-06,2.86936691501246,2.5104561407531017e-06,3.0085786792171874,5.9315731984135126e-08,4.075227375028376
16,0.0003156926325192741,1.0166380127360704,4.868139575107377e-06,2.066495021773969,1.3524670002715548e-05,1.9726061466954956,2.2637054096085415e-07,2.943375703788849,3.109352466111639e-07,3.01326345014418,3.5838316160747937e-09,4.048840033522424
32,0.00015694981569029954,1.0082172573985897,1.1914387789657804e-06,2.030665710538596,3.406655353371523e-06,1.989165488760122,2.8775140081753836e-08,2.9757913665520412,3.864362550363791e-08,3.0083118074372273,2.200132519648881e-10,4.025840526210629
64,7.82563303421939e-05,1.0040239697924704,2.9489226737235286e-07,2.014444923974544,8.54397010766823e-07,1.9953774946803833,3.624176385466381e-09,2.9890978747332966,4.815432122838617e-09,3.004493280392861,1.3627152470023461e-11,4.013034390433589
\end{filecontents*}
\pgfplotstabletypeset[
  col sep = comma,
  every head row/.style = {
    before row={
      \toprule
      & \multicolumn{4}{c}{$ k = 1 $}
      & \multicolumn{4}{c}{$ k = 2 $}
      & \multicolumn{4}{c}{$ k = 3 $}\\
      \cmidrule(lr){2-5}
      \cmidrule(lr){6-9}
      \cmidrule(lr){10-13}
    },
    after row=\midrule
  },
  every last row/.style = { after row=\bottomrule },
  every even column/.style = { fixed, fixed zerofill }, 
  every odd column/.style = { sci, sci e, sci zerofill }, 
  columns/N/.style = {
    column name = $N$,
    precision = 0,
    column type = r,
  },
  columns/{k=1_energy_error}/.style={
    column name=$ \lVert u - u _h \rVert _h $
  },
  columns/{k=1_energy_rate}/.style={
    column name=rate
  },
  columns/{k=1_l2_error}/.style={
    column name=$ \lVert u - u _h \rVert _\Omega $
  },
  columns/{k=1_l2_rate}/.style={
    column name=rate
  },
  columns/{k=2_energy_error}/.style={
    column name=$ \lVert u - u _h \rVert _h $
  },
  columns/{k=2_energy_rate}/.style={
    column name=rate
  },
  columns/{k=2_l2_error}/.style={
    column name=$ \lVert u - u _h \rVert _\Omega $
  },
  columns/{k=2_l2_rate}/.style={
    column name=rate
  },
  columns/{k=3_energy_error}/.style={
    column name=$ \lVert u - u _h \rVert _h $
  },
  columns/{k=3_energy_rate}/.style={
    column name=rate
  },
  columns/{k=3_l2_error}/.style={
    column name=$ \lVert u - u _h \rVert _\Omega $
  },
  columns/{k=3_l2_rate}/.style={
    column name=rate
  },
  ]{square.csv}
  \caption{Convergence to a smooth solution on a square
    domain.\label{tab:square}}
\end{table}

\Cref{tab:square} shows the result of applying our method to the
problem whose exact solution is
\begin{equation*}
  u =
  \begin{bmatrix*}
    ( x ^3 / 3 - x ^2 / 4  ) ( y ^2 -  y / 2 ) \sin y \\[1ex]
    ( y ^3 / 3 - y ^2 / 4 ) ( x ^2 - x / 2 ) \cos x 
  \end{bmatrix*}.
\end{equation*}
(This is the same $u$ that \citet{BrCuLiSu2008} use for their
numerical experiments on the square.) Since $u$ is smooth, we observe
convergence rates of $k$ for the energy error and $ k + 1 $ for the
$ L ^2 $ error, consistent with \cref{t:error_estimates}.

\subsection{Minimum-regularity solutions on L-shaped domain}
\label{s:min-reg}

We next consider the L-shaped domain
$ \Omega = \bigl( - \frac{1}{2}, \frac{1}{2} \bigr) ^2 \setminus
\bigl[ 0, \frac{1}{2} \bigr] ^2 $. This has a reentrant corner at the
origin with angle $ \omega _1 = 3 \pi / 2 $, so we may take any
$ \mu _1 < 1/3 $. The remaining corners $ c _\ell $ have
$ \omega _\ell = \pi / 2 $ and thus $ \mu _\ell = 1 $ for
$ \ell = 2, \ldots, 6 $. Given $ N \in \mathbb{N} $, we construct a
uniform triangle mesh of $\Omega$ by taking a uniform
$ 2 N \times 2 N $ mesh of the square
$ \bigl( - \frac{1}{2} , \frac{1}{2} \bigr) ^2 $, as in
\cref{s:square}, and removing the first quadrant.

\subsubsection{Minimum-regularity singular harmonic vector field}
\label{s:min-reg_harmonic}

In polar coordinates $ ( r, \theta ) $, we first consider the problem
whose exact solution is
\begin{equation*}
  u = \nabla \times \Biggl( r ^{ 2/3 } \cos \biggl[ \frac{ 2 }{ 3 } \Bigl( \theta - \frac{ \pi }{ 2 } \Bigr) \biggr] \Biggr),
\end{equation*}
which is a harmonic vector field with $ \nabla \cdot u = 0 $ and
$ \nabla \times u = 0 $. We observe that
$ u \in \bigl[ V ^m _{ m - 1/3 - \mu } (\Omega) \bigr] ^2 $ for all
$m$, since the condition for $ \partial ^\beta u $ to be in the
appropriate weighted $ L ^2 $ space in a $\delta$-neighborhood of the
origin is
\begin{equation*}
  \int _0 ^\delta \bigl( r ^{- 1/3 - \mu _1 + \lvert \beta \rvert } r ^{ 2/3 - 1 - \lvert \beta \rvert } \bigr) ^2 r \,\mathrm{d}r = \int _0 ^\delta r ^{ 2 ( 1/3 - \mu _1 ) - 1 } \,\mathrm{d}r < \infty ,
\end{equation*}
which holds since $ \mu _1 < 1/3 $. (Compare \citet[Theorem
6.1]{CoDa2002}.) By interpolation, we get
$ u \in \bigl[ V ^{ 1/3 + 1 } _{ 1 - \mu } (\Omega) \bigr] ^2 $, so
the hypotheses of the error estimates in \cref{s:error_estimates} hold
with $ s = 1 / 3 $.

\begin{remark}
  \label{r:bc}
  Although $u$ does not satisfy the homogeneous boundary condition
  $ u \times n = 0 $ on all of $\partial \Omega$, it does satisfy this
  condition on the boundary edges $ \theta = \pi / 2 $ and
  $ \theta = 2 \pi $ adjacent to the reentrant corner. Thus, taking
  $\phi$ to be a smooth cutoff function supported in a small
  neighborhood of the origin, we may write
  $ u = u \phi + u ( 1 - \phi ) $, where $ u \phi $ satisfies the
  homogeneous boundary condition and $ u ( 1 - \phi ) $ is a smooth
  extension of the inhomogeneous boundary condition. It follows that
  $u$ and $ u \phi $ have the same regularity, and standard arguments
  may be used to extend the numerical properties of the method from
  the homogeneous boundary value problem with exact solution $u \phi $
  to the inhomogeneous boundary value problem with exact solution
  $ u $.
\end{remark}

\cref{tab:min-reg_harmonic} shows the results of applying our method
to this problem, where the inhomogeneous boundary conditions are
imposed on $ \hat{ u } _h \times n $ by interpolating $ u \times n $
on $ \mathcal{E} _h ^\partial $. Since $ s = 1/3 $, we observe minimal
convergence rates of approximately $ 1/3 $ for the energy error and
$ 2/3 $ for the $ L ^2 $ error for all $k$, consistent with
\cref{t:error_estimates}.

\subsubsection{Minimum-regularity nonsingular vector field}
\label{s:min-reg_non-harmonic}

\begin{table}
  \centering
  \scriptsize
  \begin{filecontents*}{lsd_singular_0.667.csv}
N,k=1_energy_error,k=1_energy_rate,k=1_l2_error,k=1_l2_rate,k=2_energy_error,k=2_energy_rate,k=2_l2_error,k=2_l2_rate,k=3_energy_error,k=3_energy_rate,k=3_l2_error,k=3_l2_rate
2,0.22282646472974513,,0.07708512794838489,,0.16007456103019535,,0.05731865876159625,,0.1450461312099107,,0.04421473059613431,
4,0.17726448324724045,0.330017082872442,0.05022036069078385,0.6181801571814023,0.1243891984383118,0.36388284182348674,0.03686334560582524,0.636817839837569,0.11348880715136826,0.3539617967538078,0.028191939339679974,0.6492443733774076
8,0.14018017722889997,0.3386211531044605,0.03233742463743814,0.635067603612133,0.09727360736757297,0.3547408859715344,0.02346664117640412,0.6515756505139478,0.08919323350055153,0.3475438468437426,0.017870103254469277,0.657734754764312
16,0.11075585757497616,0.3398993523743785,0.020664246513917777,0.646068027417531,0.07644566564791598,0.3476137154753922,0.014867220474715634,0.6584764098257434,0.07033180341377025,0.34275705592432537,0.011295727434124728,0.6617707884632518
32,0.08757231947482232,0.33883617203351424,0.013136597165565062,0.6535451433027291,0.060284885896909036,0.34263835782080454,0.009396307354029422,0.6619691412975541,0.055586013484743714,0.3394552920048257,0.0071297669966294335,0.6638503471587791
64,0.06931183129381041,0.33737328733583993,0.008323080132273569,0.6584021838016545,0.047650112942411234,0.33931671380727924,0.0059307698776582315,0.6638745105558028,0.04399863784157102,0.337263058762554,0.004496685008997425,0.6649931023119304
\end{filecontents*}
\pgfplotstabletypeset[
  col sep = comma,
  every head row/.style = {
    before row={
      \toprule
      & \multicolumn{4}{c}{$ k = 1 $}
      & \multicolumn{4}{c}{$ k = 2 $}
      & \multicolumn{4}{c}{$ k = 3 $}\\
      \cmidrule(lr){2-5}
      \cmidrule(lr){6-9}
      \cmidrule(lr){10-13}
    },
    after row=\midrule
  },
  every last row/.style = { after row=\bottomrule },
  every even column/.style = { fixed, fixed zerofill }, 
  every odd column/.style = { sci, sci e, sci zerofill }, 
  columns/N/.style = {
    column name = $N$,
    precision = 0,
    column type = r,
  },
  columns/{k=1_energy_error}/.style={
    column name=$ \lVert u - u _h \rVert _h $
  },
  columns/{k=1_energy_rate}/.style={
    column name=rate
  },
  columns/{k=1_l2_error}/.style={
    column name=$ \lVert u - u _h \rVert _\Omega $
  },
  columns/{k=1_l2_rate}/.style={
    column name=rate
  },
  columns/{k=2_energy_error}/.style={
    column name=$ \lVert u - u _h \rVert _h $
  },
  columns/{k=2_energy_rate}/.style={
    column name=rate
  },
  columns/{k=2_l2_error}/.style={
    column name=$ \lVert u - u _h \rVert _\Omega $
  },
  columns/{k=2_l2_rate}/.style={
    column name=rate
  },
  columns/{k=3_energy_error}/.style={
    column name=$ \lVert u - u _h \rVert _h $
  },
  columns/{k=3_energy_rate}/.style={
    column name=rate
  },
  columns/{k=3_l2_error}/.style={
    column name=$ \lVert u - u _h \rVert _\Omega $
  },
  columns/{k=3_l2_rate}/.style={
    column name=rate
  },
  ]{lsd_singular_0.667.csv}
  \caption{Convergence to the minimum-regularity singular harmonic on
    an L-shaped domain.\label{tab:min-reg_harmonic}}

  \begin{filecontents*}{lsd_nonsingular_2.001.csv}
N,k=1_energy_error,k=1_energy_rate,k=1_l2_error,k=1_l2_rate,k=2_energy_error,k=2_energy_rate,k=2_l2_error,k=2_l2_rate,k=3_energy_error,k=3_energy_rate,k=3_l2_error,k=3_l2_rate
2,0.007233811675817748,,0.0036455377442272765,,0.0045237153307884906,,0.0019374599241847679,,0.00296296660054096,,0.0010409733428496486,
4,0.00605790079876699,0.2559380860781939,0.0025728217269905974,0.5027801436486108,0.003714180338409574,0.28446428037167415,0.001311403106069593,0.5630552519388416,0.0024031473050349045,0.3021172802497478,0.0006870722715524754,0.5993993587953419
8,0.004981876527638796,0.2821286876314799,0.001749222558883942,0.5566376355690323,0.0030155049725031275,0.3006442514536882,0.0008664347022633707,0.5979482852750593,0.001934763074865074,0.3127681744557673,0.00044612267857020347,0.6230213716676971
16,0.004046723365336218,0.29993503631212626,0.001158561687603865,0.594378996868496,0.00242888286221766,0.3121066994042341,0.0005629521663030187,0.622078684017696,0.0015496304049616239,0.3202327430512746,0.0002865148191840509,0.638830728821711
32,0.0032599031566798923,0.3119251231698241,0.0007538599636198081,0.6199664028275027,0.00194582509734374,0.3199108770728003,0.0003616983582623767,0.6382252929483769,0.0012368598265968954,0.3252421572803019,0.00018269179771255713,0.6491978973630008
64,0.002611636204858621,0.3198751596871574,0.0004847881603673486,0.636942089251237,0.0015532683107960421,0.32507497601980795,0.00023069252936709327,0.6488157644066792,0.0009849753348253067,0.3285225058464363,0.00011595023595071104,0.6559061071755039
\end{filecontents*}
\pgfplotstabletypeset[
  col sep = comma,
  every head row/.style = {
    before row={
      \toprule
      & \multicolumn{4}{c}{$ k = 1 $}
      & \multicolumn{4}{c}{$ k = 2 $}
      & \multicolumn{4}{c}{$ k = 3 $}\\
      \cmidrule(lr){2-5}
      \cmidrule(lr){6-9}
      \cmidrule(lr){10-13}
    },
    after row=\midrule
  },
  every last row/.style = { after row=\bottomrule },
  every even column/.style = { fixed, fixed zerofill }, 
  every odd column/.style = { sci, sci e, sci zerofill }, 
  columns/N/.style = {
    column name = $N$,
    precision = 0,
    column type = r,
  },
  columns/{k=1_energy_error}/.style={
    column name=$ \lVert u - u _h \rVert _h $
  },
  columns/{k=1_energy_rate}/.style={
    column name=rate
  },
  columns/{k=1_l2_error}/.style={
    column name=$ \lVert u - u _h \rVert _\Omega $
  },
  columns/{k=1_l2_rate}/.style={
    column name=rate
  },
  columns/{k=2_energy_error}/.style={
    column name=$ \lVert u - u _h \rVert _h $
  },
  columns/{k=2_energy_rate}/.style={
    column name=rate
  },
  columns/{k=2_l2_error}/.style={
    column name=$ \lVert u - u _h \rVert _\Omega $
  },
  columns/{k=2_l2_rate}/.style={
    column name=rate
  },
  columns/{k=3_energy_error}/.style={
    column name=$ \lVert u - u _h \rVert _h $
  },
  columns/{k=3_energy_rate}/.style={
    column name=rate
  },
  columns/{k=3_l2_error}/.style={
    column name=$ \lVert u - u _h \rVert _\Omega $
  },
  columns/{k=3_l2_rate}/.style={
    column name=rate
  },
  ]{lsd_nonsingular_2.001.csv}
  \caption{Convergence to a minimum-regularity nonsingular solution on
    an L-shaped domain.\label{tab:min-reg_non-harmonic}}
\end{table}

The next example shows that even a nonsingular vector field may have
minimum regularity, owing to the conditions
$ \nabla \cdot u , \nabla \times u \in V ^s _{ \mu -1 } (\Omega) $.
Given arbitrarily small $ \epsilon > 0 $, consider the problem whose
exact solution is
\begin{equation*}
  u = \nabla \times r ^{ 2 + \epsilon } ,
\end{equation*}
where the inhomogeneous boundary conditions may also be dealt with as
in \cref{r:bc}. By a similar calculation as for the singular harmonic
vector field, we have
$ u \in \bigl[ V ^m _{ m - 5/3 - \mu } (\Omega) \bigr] ^2 $ for all
$m$, and thus
$ u \in \bigl[ V ^{ 5/3 + 1 } _{ 1 - \mu } (\Omega) \bigr] ^2
$. However, we merely have
$ \nabla \times u \in V ^m _{ m - 1/3 + \mu - 1 } (\Omega) $ for all
$m$, provided that $ 1/3 - \epsilon < \mu _1 < 1/3 $, since
\begin{equation*}
  \int _0 ^\delta ( r ^{ - 1/3 + \mu _1 - 1 + \lvert \beta \rvert } r ^{\epsilon - \lvert \beta \rvert } ) ^2 r \,\mathrm{d}r = \int _0 ^\delta r ^{ 2 ( \mu _1 - 1/3 + \epsilon ) - 1 } \,\mathrm{d}r < \infty .
\end{equation*}
and interpolation gives
$ \nabla \times u \in V ^{1/3} _{ \mu -1 } (\Omega) $. Hence, even
though $u$ does not have a singularity at the origin, the regularity
hypotheses of \cref{t:error_estimates} hold merely with $ s = 1 / 3 $.

\cref{tab:min-reg_non-harmonic} shows the results of applying our method
to this problem with $ \epsilon = 0.001 $. As with the previous
example, since $ s = 1/3 $, we observe minimal convergence rates of
approximately $ 1/3 $ for the energy error and $ 2/3 $ for the
$ L ^2 $ error for all $k$, consistent with \cref{t:error_estimates}.

\subsection{Higher-regularity solutions on L-shaped domain}

Finally, we present numerical results for convergence to solutions
with higher regularity on the L-shaped domain, observing improved
convergence for larger $k$. As in \cref{s:min-reg}, we consider both a
harmonic and a non-harmonic example---here, both having
$ s = 7/3 $---on the same family of uniform meshes, where
inhomogeneous boundary conditions for $ u \times n $ on
$ \partial \Omega $ are handled in the same way.

\subsubsection{Higher-regularity harmonic vector field}
\label{s:higher_harmonic}

Consider the problem whose exact solution is
\begin{equation*}
  u = \nabla \times \Biggl( r ^{ 8/3 } \cos \biggl[ \frac{ 8 }{ 3 } \Bigl( \theta - \frac{ \pi }{ 2 } \Bigr) \biggr] \Biggr),
\end{equation*}
which is a harmonic vector field with $ \nabla \cdot u = 0 $ and
$ \nabla \times u = 0 $. By a similar calculation to that in
\cref{s:min-reg_harmonic}, we get
$ u \in \bigl[ V ^m _{ m - 7/3 - \mu } (\Omega) \bigr] ^2 $ for all
$m$. By interpolation, we see that
$ u \in \bigl[ V ^{ 7/3 + 1 }_{ 1 - \mu } (\Omega) \bigr] ^2 $, so the
hypotheses of the error estimates in \cref{s:error_estimates} hold
with $ s = 7 / 3 $.

\cref{tab:higher_harmonic} shows the results of applying our method to
this problem. Since $ s \leq 3 $, for $ k = 3 $ we observe the maximum
convergence rates predicted by \cref{t:error_estimates}: roughly
$ 7/3 $ for the energy error and $ 8/3 $ for the $ L ^2 $ error. For
$ k = 2 $, however, we \emph{also} observe rates of approximately
$ 7/3 $ for the energy error and $ 8/3 $ for the $ L ^2 $ error. This
is explained by the fact that $u$ is the curl of a harmonic function,
and $ \textit{BSM}_k (K) $ contains gradients (hence curls) of
harmonic polynomials with degree $ \leq 2k $. In this special case,
the condition $ s \leq k $ in the approximation estimate
\cref{l:energy_approx} improves to $ s \leq 2 k -1 $, while the
consistency error in \cref{l:consistency} vanishes due to
$ \nabla \cdot u = 0 $ and $ \nabla \times u = 0 $. For $ k = 1 $, we
observe the expected energy-norm convergence rate of $1$, but the
$ L ^2 $-norm convergence rate of $2$ is \emph{better} than the
duality-based estimate of $ 4/3 $ in \cref{t:error_estimates}. We do
not yet have a satisfying analytical explanation for this
better-than-expected gap between the energy-norm and $ L ^2 $-norm
rates when $ s > k $; see further discussion in the next example and
in \cref{s:conclusion}.

\subsubsection{Higher-regularity non-harmonic vector field}

\begin{table}
  \centering
  \scriptsize
  \begin{filecontents*}{lsd_singular_2.667.csv}
N,k=1_energy_error,k=1_energy_rate,k=1_l2_error,k=1_l2_rate,k=2_energy_error,k=2_energy_rate,k=2_l2_error,k=2_l2_rate,k=3_energy_error,k=3_energy_rate,k=3_l2_error,k=3_l2_rate
2,0.10792581960146652,,0.01794423721881384,,0.0008621877238775246,,0.0002276824752059249,,0.00015603613592566842,,3.84387445925222e-05,
4,0.05504433261888956,0.9713741126065965,0.004490977960865289,1.998419050438871,0.00017225064469704988,2.3234926384429,3.625809120347072e-05,2.6506483691035214,3.0725086716523605e-05,2.344391189572727,6.121017763908739e-06,2.650717758400374
8,0.0278713067815499,0.9818133845055073,0.0011189386725468102,2.004898675820799,3.41876630304489e-05,2.332961669833251,5.734613412385883e-06,2.660534838625054,6.060437862501358e-06,2.341923146065078,9.687079940909933e-07,2.659637802600244
16,0.014033164123865616,0.989940309705446,0.000278934108155961,2.0041347028998455,6.772128744281053e-06,2.3357945021484383,9.049357977412799e-07,2.66380888668116,1.1975247626220898e-06,2.3393665426313817,1.5291049739474606e-07,2.663374395353525
32,0.007042267166283243,0.9947284712570479,6.960257811598416e-05,2.002711708027732,1.3412589561613798e-06,2.3360215959218102,1.4266574607796924e-07,2.6651764547327375,2.3695414338750853e-07,2.3373756952509908,2.4109318914208735e-08,2.6650246513121942
64,0.003527719713267416,0.9973040210551547,1.7381748178747033e-05,2.001567555958916,2.6574011781445193e-07,2.3354998567267025,2.248141288304624e-08,2.66583437439051,4.693168591928027e-08,2.335973696093575,3.79927519433843e-09,2.6657947753524596
\end{filecontents*}
\pgfplotstabletypeset[
  col sep = comma,
  every head row/.style = {
    before row={
      \toprule
      & \multicolumn{4}{c}{$ k = 1 $}
      & \multicolumn{4}{c}{$ k = 2 $}
      & \multicolumn{4}{c}{$ k = 3 $}\\
      \cmidrule(lr){2-5}
      \cmidrule(lr){6-9}
      \cmidrule(lr){10-13}
    },
    after row=\midrule
  },
  every last row/.style = { after row=\bottomrule },
  every even column/.style = { fixed, fixed zerofill }, 
  every odd column/.style = { sci, sci e, sci zerofill }, 
  columns/N/.style = {
    column name = $N$,
    precision = 0,
    column type = r,
  },
  columns/{k=1_energy_error}/.style={
    column name=$ \lVert u - u _h \rVert _h $
  },
  columns/{k=1_energy_rate}/.style={
    column name=rate
  },
  columns/{k=1_l2_error}/.style={
    column name=$ \lVert u - u _h \rVert _\Omega $
  },
  columns/{k=1_l2_rate}/.style={
    column name=rate
  },
  columns/{k=2_energy_error}/.style={
    column name=$ \lVert u - u _h \rVert _h $
  },
  columns/{k=2_energy_rate}/.style={
    column name=rate
  },
  columns/{k=2_l2_error}/.style={
    column name=$ \lVert u - u _h \rVert _\Omega $
  },
  columns/{k=2_l2_rate}/.style={
    column name=rate
  },
  columns/{k=3_energy_error}/.style={
    column name=$ \lVert u - u _h \rVert _h $
  },
  columns/{k=3_energy_rate}/.style={
    column name=rate
  },
  columns/{k=3_l2_error}/.style={
    column name=$ \lVert u - u _h \rVert _\Omega $
  },
  columns/{k=3_l2_rate}/.style={
    column name=rate
  },
  ]{lsd_singular_2.667.csv}
  \caption{Convergence to a higher-regularity harmonic vector field on
    an L-shaped domain.\label{tab:higher_harmonic}}

  \begin{filecontents*}{lsd_nonsingular_4.001.csv}
N,k=1_energy_error,k=1_energy_rate,k=1_l2_error,k=1_l2_rate,k=2_energy_error,k=2_energy_rate,k=2_l2_error,k=2_l2_rate,k=3_energy_error,k=3_energy_rate,k=3_l2_error,k=3_l2_rate
2,2.6732353740517416,,0.8248603368531169,,0.5333515675097946,,0.1821426669899357,,6.84199853462649e-05,,2.1816700036932673e-05,
4,1.3143019090061705,1.0242901541472387,0.20292419329676734,2.023208988522337,0.14731211371455866,1.856210752443055,0.032945037074577736,2.466935858363751,1.41581318108524e-05,2.27278688316078,3.5832535408742385e-06,2.6060908648715504
8,0.6469387689179497,1.0225956395557891,0.050227768480623186,2.0143837927578785,0.038905845389054984,1.9208172368145915,0.0056240169246590905,2.550388309186496,2.87044568590658e-06,2.3022842494125904,5.802849753339981e-07,2.6264366477683403
16,0.32036019108191455,1.0139342860513525,0.012437871690272816,2.0137455409347087,0.010025991746399983,1.9562419769052894,0.0009283921546327527,2.598794699628737,5.758536163240759e-07,2.3175007313528795,9.30265401591145e-08,2.6410472991930902
32,0.15933873754960037,1.0075978357548803,0.003087477932766775,2.0102389090164987,0.0025500821086233342,1.9751293461271873,0.00015037383507133408,2.6261807702025948,1.1492032614250129e-07,2.3250681286275103,1.4815103050250662e-08,2.6505737096199304
64,0.07945168659591184,1.0039472990567488,0.0007684307002555619,2.006441761657177,0.0006439775111661308,1.985461487468406,2.408700147487665e-05,2.642226848357664,2.2871523782891792e-08,2.329009600761207,2.3493259933149287e-09,2.656749840136694
\end{filecontents*}
\pgfplotstabletypeset[
  col sep = comma,
  every head row/.style = {
    before row={
      \toprule
      & \multicolumn{4}{c}{$ k = 1 $}
      & \multicolumn{4}{c}{$ k = 2 $}
      & \multicolumn{4}{c}{$ k = 3 $}\\
      \cmidrule(lr){2-5}
      \cmidrule(lr){6-9}
      \cmidrule(lr){10-13}
    },
    after row=\midrule
  },
  every last row/.style = { after row=\bottomrule },
  every even column/.style = { fixed, fixed zerofill }, 
  every odd column/.style = { sci, sci e, sci zerofill }, 
  columns/N/.style = {
    column name = $N$,
    precision = 0,
    column type = r,
  },
  columns/{k=1_energy_error}/.style={
    column name=$ \lVert u - u _h \rVert _h $
  },
  columns/{k=1_energy_rate}/.style={
    column name=rate
  },
  columns/{k=1_l2_error}/.style={
    column name=$ \lVert u - u _h \rVert _\Omega $
  },
  columns/{k=1_l2_rate}/.style={
    column name=rate
  },
  columns/{k=2_energy_error}/.style={
    column name=$ \lVert u - u _h \rVert _h $
  },
  columns/{k=2_energy_rate}/.style={
    column name=rate
  },
  columns/{k=2_l2_error}/.style={
    column name=$ \lVert u - u _h \rVert _\Omega $
  },
  columns/{k=2_l2_rate}/.style={
    column name=rate
  },
  columns/{k=3_energy_error}/.style={
    column name=$ \lVert u - u _h \rVert _h $
  },
  columns/{k=3_energy_rate}/.style={
    column name=rate
  },
  columns/{k=3_l2_error}/.style={
    column name=$ \lVert u - u _h \rVert _\Omega $
  },
  columns/{k=3_l2_rate}/.style={
    column name=rate
  },
  ]{lsd_nonsingular_4.001.csv}
  \caption{Convergence to a higher-regularity non-harmonic vector
    field on an L-shaped domain.\label{tab:higher_non-harmonic}}
\end{table}

Given arbitrarily small $ \epsilon > 0 $, consider the problem whose
exact solution is
\begin{equation*}
  u = \nabla \times r ^{ 4 + \epsilon } .
\end{equation*}
By a similar calculation to that in \cref{s:min-reg_non-harmonic}, we
have $ u \in \bigl[ V ^m _{ m - 11/3 - \mu } (\Omega) \bigr] ^2 $ for
all $m$, and thus
$ u \in \bigl[ V ^{ 11/3 + 1 } _{ 1 - \mu } (\Omega) \bigr] ^2
$. However, we merely have
$ \nabla \times u \in V ^m _{ m - 7/3 + \mu -1 } (\Omega) $ for all
$m$, provided that $ 1/3 - \epsilon < \mu _1 < 1/3 $, so interpolation
gives $ \nabla \times u \in V ^{ 7/3 } _{ \mu -1 } (\Omega) $. Hence,
the regularity hypotheses of \cref{t:error_estimates} hold merely with
$ s = 7/3 $.

\cref{tab:higher_non-harmonic} shows the results of applying our
method to this problem with $ \epsilon = 0.001 $. For all $k$, we
observe a convergence rate of approximately $\min(k, 7/3)$ in the energy norm,
consistent with \cref{t:error_estimates}. This also supports the
argument that the improved energy error in \cref{s:higher_harmonic},
which is not observed here, was due to that exact solution being the
curl of a harmonic function. For $ k = 3 $, we observe the expected
$ L ^2 $-norm convergence rate of approximately $ 8 / 3 $. However,
for $ k = 2 $ and $ k = 1 $, we observe better-than-expected rates of
$ 8/3 $ (rather than $ 7/3 $) and $ 2 $ (rather than $4/3$),
respectively, similar to what we saw with the $ k = 1 $ case in
\cref{s:higher_harmonic}.

\section{Conclusion}
\label{s:conclusion}

We have presented a nonconforming primal hybrid finite element method
for the two-dimensional vector Laplacian that extends the
$ P _1 $-nonconforming method of \citet{BrCuLiSu2008} to arbitrary
order $k$.  The method uses only standard polynomial finite elements,
although the more exotic BSM element and projection play a key role in
the analysis, and the method may be implemented efficiently using
static condensation. Using the weighted Sobolev spaces of
\citeauthor{Kondratiev1967} for domains with corners, we have obtained
error estimates that hold on general shape-regular meshes, without
mesh-grading conditions. These estimates establish the convergence of
the method, even for minimum-regularity solutions with corner
singularities, and the convergence rate improves with $k$ to the
extent regularity allows.

Let us conclude with a brief discussion of one area where the
numerical experiments in \cref{s:experiments} suggest possible room
for improvement. Dropping the hypothesis that $ s \leq k $, we may
rewrite the estimates of \cref{t:error_estimates} as
\begin{align}
  \lVert u - u _h \rVert _h &\lesssim h ^{\min(k,s)} \bigl( \lVert u \rVert _{ s + 1 , 1 - \mu } + \lVert \nabla \cdot u \rVert _{ s, \mu -1 } + \lVert \nabla \times u \rVert _{ s, \mu -1 } \bigr) , \label{e:energy_estimate}\\
  \lVert u - u _h \rVert _\Omega &\lesssim h ^{ \min(k,s) + t } \bigl( \lVert u \rVert _{ s + 1 , 1 - \mu } + \lVert \nabla \cdot u \rVert _{ s, \mu -1 } + \lVert \nabla \times u \rVert _{ s, \mu -1 } \bigr) \label{e:l2_estimate}.
\end{align}
From the numerical experiments, it appears that the energy estimate
\eqref{e:energy_estimate} is sharp, and in general one cannot relax
the restriction $ t < \mu_{\min} $ in the $ L ^2 $ estimate
\eqref{e:l2_estimate}. However, when $ s > k $, it appears that a
sharper estimate than \eqref{e:l2_estimate} holds, which we now state
as a conjecture.

\begin{conjecture}
  Under the conditions of \cref{t:error_estimates} (except for the
  condition $ s \leq k $),
  \begin{equation*}
    \lVert u - u _h \rVert _\Omega \lesssim h ^{ \min (k + 1 , s + t ) } \bigl( \lVert u \rVert _{ s + 1 , 1 - \mu } + \lVert \nabla \cdot u \rVert _{ s, \mu -1 } + \lVert \nabla \times u \rVert _{ s, \mu -1 } \bigr).
  \end{equation*} 
\end{conjecture}
Establishing this would require some new analytical arguments, perhaps
involving weighted-norm error estimates. Indeed, a duality estimate of
the sort in \cref{l:duality} can only give an $ L ^2 $ error estimate
of the form \eqref{e:l2_estimate}, where the improvement over the
energy rate is the same for all $k$, and the numerical experiments
suggest that there is no way to sharpen this uniformly in $k$.

Finally, it is natural to ask how the two-dimensional method presented
in this paper might be generalized to three dimensions. In contrast
with some other approaches that are limited to dimension two---such as
methods that use the Hodge decomposition to transform vector problems
into scalar problems
\citep{BrCuNaSu2012,LiZh2015,LiZh2017,BrGeSu2017}---the variational
form of the hybrid method \eqref{e:hybrid} extends naturally to the
three-dimensional case. The main challenge is to choose suitable
finite element spaces and a suitable penalty, and here there are two
obstacles to overcome. First, when $ k > 1 $, we do not yet know a
three-dimensional version of the BSM element and commuting projection,
which would be needed to make the analysis work. (A naive extension to
three dimensions fails to satisfy unisolvence when $ k = 2 $, as shown
by \citet{Mirebeau2012}.)  Consequently, it is not clear what
polynomial degrees would be needed for the finite element spaces
$ V _h $, $ Q _h $, and $ \hat{V} _h $. Second, the weighted Sobolev
analysis in three-dimensional domains becomes more complicated, since
singularities can form along boundary edges, as well as at corners
where edges meet, cf.~\citet{CoDa2002}. Consequently, a penalty
$\gamma$ would need to be carefully constructed, likely involving the
distances both to edges and to corners with suitable exponents.

\section*{Acknowledgments}

We are grateful to the anonymous referees for their helpful comments
and suggestions.

\end{document}